\documentclass[reqno,final,12pt]{article}
\usepackage{amssymb,amsthm,amsmath,color,url,fullpage,cite,stmaryrd
}
\usepackage{url} 
\usepackage{hyperref}
\usepackage[noabbrev,capitalise]{cleveref}
\usepackage{tikzit}

\tikzstyle{Black Dot}=[fill=black, draw=black, shape=circle]
\tikzstyle{Empty}=[fill=white, draw=black, shape=circle]

\tikzstyle{Red}=[-, draw=red, line width=1.2pt]
\tikzstyle{Blue}=[-, draw=blue, line width=1.0pt]
\tikzstyle{Border}=[-, draw=red, dashed]
\tikzstyle{Dashed}=[-, dashed]

\newtheorem{thm}{Theorem}[section]
\newtheorem{lem}[thm]{Lemma}
\newtheorem{obs}[thm]{Observation}

\newtheorem{cor}[thm]{Corollary}

\newtheorem{qn}[thm]{Question}
\newtheorem{problem}[thm]{Problem}

\newcommand{\eps}{\varepsilon}

\newcommand{\mc}[1]{\mathcal{#1}}
\newcommand{\bb}[1]{\mathbb{#1}}
\newcommand{\dd}[1]{\partial{#1}}
\newcommand{\brm}[1]{\operatorname{#1}}

\newcommand{\nm}[1]{\llbracket #1 \rrbracket}

\newcommand{\dH}{{\partial H}}
\newcommand{\dlH}{{\partial_{l} H}}
\newcommand{\drH}{{\partial_{r} H}}
\newcommand{\F}{{\mathcal{F}}}
\newcommand{\tw}{{\operatorname{tw}}}

\begin{document}
\title{Densities of minor-closed graph classes are rational}
\author{Rohan Kapadia \thanks{Department of Computer Science and Software Engineering, Concordia University, Montreal, Quebec, Canada. Email: {\tt rohan.f.kapadia@gmail.com}} \and Sergey Norin\thanks{Department of Mathematics and Statistics, McGill University. Email: {\tt sergey.norin@mcgill.ca}. Supported by an NSERC Discovery grant.} }

\maketitle

\begin{abstract}
For a graph class $\mc{F}$, let $ex_{\mc{F}}(n)$ denote the maximum number of edges in a graph in $\mc{F}$ on $n$ vertices.	
We show that for every proper minor-closed graph class $\mc{F}$ the function $ex_{\mc{F}}(n) - \Delta n$ is eventually periodic, where $\Delta = \lim_{n \to \infty} ex_{\mc{F}}(n)/n$ is the \emph{limiting density} of $\mc{F}$. This confirms a special case of a conjecture by Geelen, Gerards and Whittle.  In particular, the limiting density of every proper minor-closed graph class is rational, which answers a question of Eppstein. 

As a major step in the proof we show that every proper minor-closed graph class contains a  subclass of bounded pathwidth with the same limiting density, confirming a conjecture of the second author. 

Finally, we investigate the set of limiting densities of classes of graphs closed under taking topological minors.
\end{abstract}

\section{Introduction}
All graphs in this paper are finite and simple. A graph $H$ is \emph{a minor} of a graph $G$ if a graph isomorphic to $H$ can be obtained from a subgraph of $G$ by contracting edges, and we say that $G$ is \emph{$H$-minor-free}, otherwise. A class of graphs $\mc{F}$ is  \emph{minor-closed} if for every $G \in \mc{F}$ and every minor  $H$ of $G$ we have $H \in \mc{F}$. \footnote{In particular, minor-closed graph classes are closed under isomorphisms.}

This paper  investigates the extremal properties of minor-closed graph classes. More precisely, we study the extremal functions and limiting densities defined as follows. The \emph{extremal function} $ex_{\mc{F}}: \bb{Z}_{+} \to \bb{Z}_{+}$ of a graph class $\mc{F}$ is defined by $$ex_{\mc{F}}(n) = \max_{G \in \mc{F}, |V(G)|=n}{|E(G)|}.$$ 
The \emph{density} of a   non-null graph $G$ is $d(G) = \frac{|E(G)|}{|V(G)|}$. 
The \emph{limiting density} of a graph class $\mathcal{F}$,  denoted by $d(\mathcal{F})$, is the minimum real number $d$ such that any  graph in $G \in \mathcal{F}$ has density at most $d + o_{|V(G)|}(1)$. 
Equivalently,
\[ d(\mathcal{F}) = \limsup_{n \rightarrow \infty} \frac{ex_{\mc{F}}(n)}{n}. \] 	
For a graph $H$, let $\brm{Forb}(H)$ denote the minor-closed class of all $H$-minor-free graphs $G$.

Limiting densities of classes $\brm{Forb}(K_t)$ have been intensively  
investigated starting with works of Wagner~\cite{Wagner64} and Mader~\cite{Mader67,Mader68}, partially due to the connections with the famous Hadwiger's conjecture~\cite{Had43}.\footnote{If $H$ is connected graph then $\mc{F}=\brm{Forb}(H)$ is closed under disjoint unions, implying that in this case the limiting density of $\mc{F}$ can be defined simply as $d(\mc{F}) = \sup{G \in \mc{F}}d(\mc{G})$.}
In~\cite{Mader67} Mader proved that $d(\brm{Forb}(K_t))$ is finite for every $t$, implying that  $d(\mathcal{F})$ is finite  for every proper minor-closed class $\mc{F}$.\footnote{A graph class is \emph{proper} if it does not include all graphs.}  Improving Mader's result, Kostochka~\cite{Kostochka82,Kostochka84} and Thomason~\cite{Thomason84} independently determined the approximate order of magnitude of  $d(\brm{Forb}(K_t))$,  and, eventually, Thomason~\cite{Thomason01} established the precise asymptotics. We refer the reader to~\cite{Thomason06} for a survey of the limiting densities of $\brm{Forb}(H)$ for general graphs $H$, and to~\cite{ThoWal19} for an update on the current state of the area.

In a different direction, Eppstein~\cite{Eppstein10} initiated the study of the set $\mc{D}$ of all limiting densities of proper minor-closed classes of graphs.  He proved that  $\mc{D}$  is countable, well-ordered and closed, and gave an explicit description  $\mc{D} \cap [0,3/2]$. McDiarmid and Przykucki~\cite{McdPrz19} further extended the results of~\cite{Eppstein10}, giving a complete description of $\mc{D} \cap [0,2]$. This paper continues the investigation in this direction. 

\vskip 5pt Our main results are Theorems~\ref{t:rational}-\ref{t:topological} below.

First, we answer a question of Eppstein~\cite{Eppstein10} showing that $\mc{D} \subseteq \bb{Q}.$

\begin{thm}\label{t:rational} 	The limiting density of every proper minor-closed class is rational. 
\end{thm}

As the crucial ingredient in the proof of \cref{t:rational}, we establish a conjecture of the second author (see~\cite[Conjecture 7.8]{GeeGerWhi13}) that the limiting density of every proper minor-closed class of graphs $\mc{F}$ is achieved by a subclass of $\mc{F}$ with a fairly simple structure. 

\begin{thm}\label{t:pathwidth0}
	For each proper minor-closed class $\mathcal{F}$ there exists a minor-closed class $\mc{F}' \subseteq \mc{F}$ of bounded pathwidth such that $d(\mc{F}')=d(\mc{F})$.\footnote{A minor-closed class of graphs has \emph{bounded pathwidth} if it does not include all trees.}
\end{thm}

Strengthening \cref{t:rational} we prove that the extremal function of every proper-minor closed class of graphs is a sum of a linear function and an eventually periodic function. This establishes a special case of a conjecture by Geelen, Gerards and Whittle~\cite[Conjecture 7.7]{GeeGerWhi13}  made  more generally for linearly dense minor-closed classes of matroids. 

\begin{thm}\label{t:periodic}
	For every proper minor-closed class $\mathcal{F}$ there exist integers $P$ and $M$ and rational numbers  $a_0, \ldots, a_{P-1}$ such that $ex_\mathcal{F}(n) = d(\mc{F}) n + a_i$ for all $n \equiv i \pmod P$ such that $n > M$.
\end{thm}

Previously, the first author~\cite{Kap18} proved analogues of Theorems~\ref{t:rational}--\ref{t:periodic} for classes of matroids of bounded branchwidth representable over a finite field. Our argument shares some ingredients with~\cite{Kap18}. However, in an attempt to make the paper more accessible, and as self-contained as possible, we prove graph theoretical versions of these results, instead of referencing  matroid theoretical statements from~\cite{Kap18}.

\vskip 5pt
The proofs of Theorems~\ref{t:rational}--\ref{t:periodic} rely on several heavy tools from structural graph minor theory:
\begin{itemize}
	\item The Robertson-Seymour graph minor theorem~\cite{GM20}, which states that graphs are well-quasi-ordered by the minor relation. More precisely we use the extension of this theorem to labeled graphs from~\cite{GMXXIII};
	\item A recent asymmetric grid  theorem due to Geelen and Joeris~\cite{GeeJoe16};
	\item An asymmetric version of the Robertson-Seymour Flat Wall theorem~\cite{GM13},  which is extracted from the  new proof of the Flat Wall theorem by Chuzhoy~\cite{Chuzhoy14}.
\end{itemize}

The first of these theorems is especially advanced, but we believe it to be essential.\footnote{The result of Eppstein~\cite{Eppstein10} that $\mc{D}$ is countable already relies on the Robertson-Seymour graph minor theorem.} In support of this belief, we investigate a closely related setting of topologically minor-closed classes, where the well-quasi-ordering is notably absent, and show that the set of their densities behaves very differently. 

A graph $H$ is \emph{a topological minor} of a graph $G$ if a graph isomorphic to $H$ can be obtained from a subgraph of $G$ by repeatedly contracting edges incident with vertices of degree two. Mader~\cite{Mader67} proved that the limiting densities of proper graph classes closed under taking topological minors are finite. Let $\mc{D}_T$ denote the set of these densities. Ne\v{s}et\v{r}il\cite{NesPrivate}(see also \cite[Question 8.9]{Norin15}) asked if \cref{t:rational} can be extended to classes of graphs closed under taking topological minors, i.e. whether $\mc{D}_T \subseteq \bb{Q}$? Our last theorem gives a negative answer to this question.

\begin{thm}\label{t:topological} 
	We have
	\begin{description}
		\item[(i)] $\mc{D}_T \cap [0,3/2]$ is well-ordered,
		\item[(ii)] $[3/2,+ \infty) \subseteq \mc{D}_T$.
	\end{description} 
\end{thm}

A recent theorem of Liu and Thomas~\cite{LiuTho20}, which we use to prove \cref{t:topological} (i), implies that every class of graphs $\mc{F}$ closed under taking topological minors with $d(\mc{F}) < 3/2$ is  well-quasi-ordered by the topological minor relation. Thus the sharp change of behavior of $\mc{D}_T$ at $3/2$ is directly tied to well-quasi-ordering.

\vskip 5pt
The rest of the paper is structured as follows. In \cref{s:patch} we briefly outline the proof of  Theorems~\ref{t:rational}--\ref{t:periodic}, introduce terminology needed to implement our strategy, and using this terminology state   Theorems~\ref{t:pathwidth} and ~\ref{t:struct2}, the first of which is a common strengthening of Theorems~\ref{t:rational} and~\ref{t:pathwidth0}, and the second is our main structural tool.  
We derive Theorems~\ref{t:pathwidth} and~\ref{t:periodic} from \cref{t:struct2} in \cref{s:proofs}. In  \cref{s:struct} we prove  \cref{t:struct2} modulo a further technical statement (\cref{t:large}) corresponding to the large treewidth case of \cref{t:struct2}.  \cref{t:large} is proved in \cref{s:large},  thus finishing the proof of Theorems~\ref{t:rational}--\ref{t:periodic}. \cref{s:top} is devoted to the proof of \cref{t:topological}. \cref{s:remarks} concludes the paper with a few remarks and open questions. 
  
\subsection*{Notation}
We denote $\{1,2,\ldots,n\}$ by $[n]$. For integers $m<n$ we denote $\{m,m+1,\ldots,n\}$ by $[m,n]$, when there is no danger of confusion between this set and the corresponding set of real numbers.  
 
We use largely standard graph-theoretical notation. 
Let $G$ be a graph, and let $X,Y \subseteq V(G)$. We denote by $G[X]$ the subgraph of $G$ induced by $X$.

A path $P$ in $G$ is an \emph{$(X,Y)$-path} if one end of $P$ is in $X$ and the other in $Y$. A \emph{linkage} in $G$ is a collection of pairwise vertex-disjoint paths. An \emph{$(X,Y)$-linkage} is a linkage $\mc{P}$ such that every path in $\mc{P}$ is an $(X,Y)$-path. Let $\kappa_G(X,Y)$  denote the maximum size of an $(X,Y)$-linkage in $G$. We write  $\kappa(X,Y)$, instead of $\kappa_G(X,Y)$,  when the choice of the graph $G$ is clear from the context.

	A \emph{separation} of a graph $G$ is a pair $(A,B)$ such that $A \cup B = V(G)$ and no edge of $G$ has one end in $A-B$ and the other in $B-A$. The \emph{order} of  a separation $(A,B)$ is $|A \cap B|$. Recall that by Menger's theorem  $\kappa_G(X,Y)$  is equal to the minimum order of a separation $(A,B)$ of $G$ such that $X \subseteq A$, $Y \subseteq B$.

We write $H \leq G$ if a graph $H$ is a minor of a graph $G$. 

\section{Patches and the statement of the main structural result}\label{s:patch}

The strategy of the proof of Theorems~\ref{t:rational}-\ref{t:periodic} can be briefly  summarized as follows. 

Let $\mc{F}$
be a proper minor-closed class of graphs, and let $\Delta=d(\mc{F})$. Consider a large graph $G$ in $\mc{F}$ such that the ``density''\footnote{In the proof we use a non-standard notion of density, hence the quotation marks.} of $G$ can not be increased by performing a bounded number of deletions and contractions. (\cref{l:pruned} below shows that we can  find such $G$.)  We find a large collection of bounded size ``patches'' $H_1,H_2,\ldots,H_k$ in $G$. The choice of $G$ implies that each of these pieces has ``density'' at least $\Delta$, as otherwise contracting them will make $G$ denser, and a careful choice of such a collection allows us to glue these patches together in a linear fashion to obtain a large graph in $\mc{F}$ with density close to $\Delta$ and bounded pathwidth, as desired in \cref{t:pathwidth0}. Meanwhile, \cref{t:periodic} is derived by analysis of similar collections of patches in  graphs that achieve $ex_{\mc{F}}(n)$. 

\subsection*{Patches and their minors}

We start implementing the strategy outlined above by formalizing the notion of a patch. This notion was originally introduced by Thomas and the second author to investigate the structure of large $t$-connected graphs in proper minor-closed classes.\footnote{Interestingly, very similar objects  called \emph{tiles} were also considered by Dvo\v{r}\'ak and Mohar\cite{DvoMoh16} in their study of crossing-critical graphs.}  Several definitions and observations given in this subsection have previously appeared in a survey ~\cite{Norin15} by the second author. 

\vskip 5pt

Let $q$ be a non-negative integer. A {\em $q$-patch} (or simply a {\em patch}) $H$ is a triple $(\nm{H},a_H,b_H)$, where $\nm{H}$ is a graph, and  $a_{H} , b_{H}: [q] \to V(\nm{H})$ are injective maps. For brevity, from now on we refer to the vertices and edges of $\nm{H}$
as vertices and edges of $H$,  and use $V(H)$ and $E(H)$ to denote their sets. Occasionally, it will be convenient for us to think of graphs as $0$-patches.

Two  $q$-patches $H$ and $H'$ are \emph{isomorphic} if there exists an isomorphism $\psi: V(H) \to V(H')$ between $\nm{H}$ and $\nm{H'}$, such that $\psi (a_{H}(i)) = a_{H'}(i)$ and $\psi(b_{H}(i)) = b_{H'}(i)$ for every $i \in [q]$.  \footnote{We generally do not distinguish between isomorphic patches unless we are considering their embedding in graphs.} We define the \emph{the left boundary of $H$} as $\dlH = a_H([q])$,  \emph{the right boundary of $H$}  as $\drH = b_H([q])$, and  \emph{the boundary of $H$} as $\dH = \dlH \cup \drH$. It is convenient to think of the functions $a_H$ and $b_H$ as a labeling of the boundary of $H$, which determines the way in which patches can be glued together.

The gluing is formalized as the following product operation.
Let $H_1$ and $H_2$ be $q$-patches. Replacing $H_1$ and $H_2$ by isomorphic patches if necessary we assume that $V(H_1) \cap V(H_2)= \emptyset.$ We define $H_1 \times H_2$ to be the $q$-patch, such that $\nm{H_1 \times H_2}$ is obtained from  $\nm{H_1} \cup \nm{H_2}$ by identifying the vertices $b_{H_1}(i)$ and $a_{H_2}(i)$ for all $i \in [q]$, and setting $a_{H_1 \times H_2}:=a_{H_1}$, $b_{H_1 \times H_2}:=b_{H_2}$. Thus $H_1 \times H_2$ is a patch obtained from ``gluing'' the left boundary of $H_2$ to the right boundary of $H_1$. Note that  $H_1 \times H_2$ is defined only up to isomorphism. Note further that
\begin{equation}\label{e:vproduct}
 |V(H_1 \times H_2)| = |V(H_1)|+|V(H_2)|-q.
\end{equation}

We say that a $q$-patch $H$ is \emph{degenerate} if $|V(H)|=q$, and we say that $H$ is \emph{non-degenerate}, otherwise.
Let $I_q$ be a degenerate $q$-patch satisfying $E(I_q)=\emptyset$ and $a_{I_q}(i) = b_{I_q}(i)$ for all $i \in [q]$. Note that $I_q$ acts as an identity with respect to the product of isomorphism classes of patches.

\vskip 10pt Let $\mc{F}$ be a graph class. 
We say that a $q$-patch $H$ is \emph{$\mc{F}$-tame} if $\nm{H^n} \in \mc{F}$ for every positive integer $n$.
Using the above notation we state a common strengthening of Theorems~\ref{t:rational} and~\ref{t:pathwidth0} 

\begin{thm}\label{t:pathwidth}
	For every proper minor-closed class $\mathcal{F}$ there exists an integer $q \geq 0$ and an $\mc{F}$-tame $q$-patch $H$ such that $\lim_{n \to \infty}d(\nm{H^n})=d(\mc{F})$. 
\end{thm}

\begin{figure}
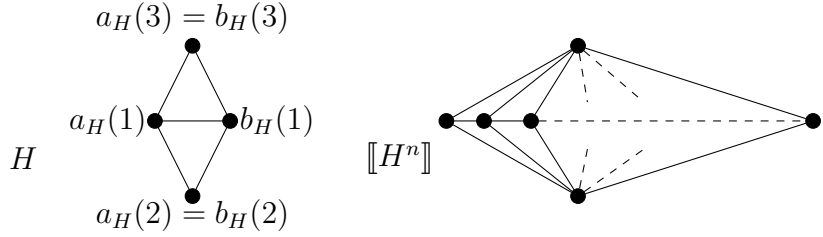

	\ctikzfig{Product}	
	\caption{A patch satisfying \cref{t:pathwidth} for the class of planar graphs.}
	\label{f:planar}
\end{figure}

Informally,
\cref{t:pathwidth} states that the limiting density of $\mc{F}$ can be achieved by repeatedly gluing copies of a fixed graph in a linear fashion. For example, if $\mc{F}$ is the class of planar graphs then $d(\mc{F})=3$, and a $3$-patch with four vertices and five edges shown on Figure~\ref{f:planar} is an $\mc{F}$-tame patch satisfying \cref{t:pathwidth} in this case.

It is easy to see that $|E(H^n)|$ is a linear function of $n$, as is $|V(H^n)|$. Thus $\lim_{n \to \infty}d(\nm{H^n})$ is rational, and so \cref{t:pathwidth} implies \cref{t:rational} as claimed. To see that \cref{t:pathwidth} implies \cref{t:pathwidth0}, note that the pathwidth of the graph $\nm{H^n}$ is at most $|V(H)|-1$ for any $n$.  Thus, for a $q$-patch $H$ as in \cref{t:pathwidth} the class of all graphs $G$ such that $G \leq \nm{H^n}$ for some $n \in \bb{N}$ satisfies \cref{t:pathwidth0}.

\vskip 10pt
Given graphs $H$ and $G$ we say that a map  $\alpha$ mapping vertices of $H$ to disjoint subsets of $V(G)$  is a \emph{model of $H$ in $G$} if \begin{enumerate}
	\item $G[\alpha(v)]$ is connected for every $v \in V(H)$;
	\item if $v,w \in V(H)$ are adjacent then some edge of $G$ joins a vertex of $\alpha(v)$ to a vertex in $\alpha(w)$.
\end{enumerate}
It is well known that a graph $H$ is a minor of $G$ if and only if there exists a model of $H$ and $G$.

We now extend the definition of the minor relation to patches using models.
We say that a map  $\alpha$  is a \emph{model of a $q$-patch $H$ in a $q$-patch $G$} if  
\begin{enumerate}
	\item $\alpha$ is the model of  $\nm{H}$ in $\nm{G}$, and
	\item $a_{G}(i) \in \alpha(a_{H}(i))$ and $b_{G}(i) \in \alpha(b_{H}(i))$ for every $i \in [q]$.
\end{enumerate}

We say that a patch $H$ is \emph{a minor of a patch $G$} and write $H \leq G$ if there exists a model of $H$ in $G$. 
It is easy to verify that the minor relation is transitive, and that if $H_1,H_2,G_1,G_2$ are patches, such that $H_1 \leq G_1$ and $H_2 \leq G_2$, then $H_1 \times H_2 \leq G_1 \times G_2$. Thus the set of isomorphism classes of patches forms an ordered monoid with respect to the product operation and minor order.

In~\cite{GMXXIII} Robertson and Seymour extend their celebrated Graph Minor theorem to graphs with vertices labeled from a well-quasi order. Their result immediately implies that the minor order is a well order on the set of  isomorphism classes of $q$-patches, as we next show. For simplicity, we only state a very specialized version of the result from~\cite{GMXXIII}.

\begin{thm}[\protect{\cite[1.7]{GMXXIII}}]\label{t:RSwqo} Let $L$ be a finite set. For every positive integer $k$, let $G_k$ be a graph, and let $\phi_k : V(G_k) \to L$ be a map. Then there exists $l>k \geq 1$ and a model $\alpha$ of $G_k$ in $G_l$ such that
	\begin{description}
	\item[(RS)]	For every $v \in V(G_k)$ there exists $w \in \alpha(v)$ such that $\phi_l(w) = \phi_k(v)$. 
	\end{description}
\end{thm}

\begin{cor} \label{t:wqo}
	For every sequence of $q$-patches $(G_k)_{k \in \bb{N}}$ there exist $l > k \geq 1$  such that $G_k \leq G_l$.
\end{cor}

\begin{proof} 
	Let $L = \{0\} \cup \{(i,j)\}_{i \in [q],j \in [3]},$ and define $\phi_k:V(G_k) \to L$ by setting
	\begin{itemize} 
	 \item	$\phi_k(a_{G_k}(i))=(i,1)$,  $\phi_k(b_{G_k}(i))=(i,2)$, if $a_{G_k}(i) \neq b_{G_k}(i)$, and $\phi_k(a_{G_k}(i))=(i,3)$, if $a_{G_k}(i) = b_{G_k}(i)$;
	\item $\phi_k(v)=0$ for every $v \in V(G_k) - \dd{G}_k$.
	\end{itemize}
	By \cref{t:RSwqo} there exist $l > k \geq 1$  and a model $\alpha$ of $\nm{G_k}$ in $\nm{G_l}$  satisfying (RS). It follows that $\alpha$ is a model of $G_k$ in $G_l$, as desired.
\end{proof}

A $q$-patch $H$ is \emph{linked} if $\nm{H}$ contains a linkage $\mc{P}=(P_i)_{i \in [q]}$  such $P_i$ has ends $a_H(i)$ and $b_H(i)$ for every  $i \in [q]$. We say that a linkage with the above properties is a \emph{linkage of $H$}. Note that a $q$-patch $H$ is linked if and only if $I_q \leq H$. In particular, for every pair of $q$-patches $H,H'$ such that $H'$ is linked, we have $H \leq H \times H'$ and $H \leq H' \times H$. We say that a patch $H$ is \emph{refined} if it is non-degenerate and linked.
 
The following lemma is our first application of \cref{t:wqo} en route to \cref{t:pathwidth}.

\begin{lem} \label{l:multiply}
	For every proper  minor-closed graph class $\mc{F}$ and every integer $q \geq 0$ there exists a positive integer $K=K_{\ref{l:multiply}}(\mc{F},q)$ such that if a $q$-patch $H$ is linked and  $\nm{ H^K} \in \mc{F}$  then $H$ is $\mc{F}$-tame.
\end{lem}

\begin{proof}
	Suppose the lemma fails for some $\mc{F}$ and $q$. Then there exists an increasing sequence of integers $(n_i)_{i \in \bb{N}}$ and a sequence of linked
	$q$-patches $(H_i)_{i \in \bb{N}}$ such that $\nm{H_i^{n_i}} \in \mc{F}$ and $\nm{H_i^{n_{i+1}}} \not \in \mc{F}$.
	It follows that $H_i \not \leq H_j$ for all integers $1 \leq i < j$, in contradiction with \cref{t:wqo}.
\end{proof}

For a patch $H$ let $$e(H)=|E(H)| - |E(\nm{H}[\dlH])|.$$ 
Thus $e(H)$ counts the number of edges of the patch $H$, except for the edges induced by the left boundary. We dismiss those edges, as they may overlap preexisting edges once $H$ is glued to another patch. The following observation is immediate.  

\begin{obs} \label{o:edgesum} 
	Let $H_1,H_2,\ldots, H_k$ be $q$-patches.
	Then $$e(H_1 \times H_2 \times \ldots \times H_k) \geq \sum_{i=1}^{k}e(H_i). $$
\end{obs}

\begin{cor}\label{c:limitdensity}
	Let $H$ be a non-degenerate $q$-patch then
	\begin{equation}\label{e:powerdensity}
	\lim_{n \to \infty}d(\nm{H^n}) \geq \frac{e(H)}{|V(H)|-q}.
	\end{equation}
\end{cor}	

\begin{proof}
Using \cref{o:edgesum} and 	\eqref{e:vproduct} we obtain
$$ d(\nm{H^n}) = \frac{|E(H^n)|}{|V(H^n)|} \geq \frac{n \cdot e(H)}{n|V(H)|- (n-1)q} = \frac{e(H)}{|V(H)|-q + \frac{q}{n}},$$
implying the corollary.
\end{proof}

\subsection*{Embedding patches and patchworks}

We are primarily interested in patches obtained by ``cutting out'' subgraphs of a large graph along their boundary. To formalize this, 
we say that a patch $H$ is \emph{embedded} in a graph $G$ if $\nm{H}=G[X]$ for some $X \subseteq V(G)$, and 
for every $v \in V(G)-X$ either all the neighbors of $v$ in  $X$ lie in $\dlH$,  or they all lie in $\drH$. In particular, if $H$ is embedded in $G$ then  $(V(H), V(G) - (V(H) - \dH))$ is a separation of $G$. 

Let $H$ be a linked $q$-patch embedded in a graph $G$. We denote by $G \slash H$   the graph obtained from $G$ by deleting the vertices in $V(H) - \dH$, the edges in $E(H) - E(G[\dlH])$, and identifying $a_H(i)$ with $b_H(i)$  for every $i \in [q]$. Note that $G \slash H \leq G$, as one can obtain $G  \slash H$ from $G$ by contacting a linkage of $H$ prior to deleting its vertices and edges.

\begin{lem}\label{l:contract1} Let $H$ be a linked $q$-patch embedded in $G$. Then $$|E(G \slash H)| = |E(G)|-e(H).$$ 	
\end{lem}
\begin{proof}
 Let $G' = G \slash H$. By definition $e(H)$ edges are deleted to obtain $G'$ from $G$, and so it suffices to check that no parallel edges are created during the identification of $\dlH$ and $\drH$. Indeed, if an edge $e \in E(H)$ is not deleted then both of its ends are in $\dlH$, and so no two such edges are identified. Moreover, a vertex $v \in V(G)-V(H)$ can not have neighbors both in $\dlH -\drH$, and in $\drH -\dlH$. Thus a pair of edges incident to such a vertex can not become parallel in the course of identification either. 
\end{proof}

Our proof works not just with individual patches embedded in graph, but with their ordered sequences, defined as follows.
A \emph{$q$-patchwork} (or simply a \emph{patchwork}) is an ordered sequence of $q$-patches $\mc{H}= ({H}_1,{H}_2,\ldots,{H}_{l})$. The \emph{vertex set} $V(\mc{H})$ of the $q$-patchwork $\mc{H}$ is defined by $V(\mc{H}) = \cup_{i=1}^{l}V(H_i)$.
For $J \subseteq [l]$ we say that the $q$-patchwork $\mc{H}|_{J}=(H_j)_{j \in J}$ is the \emph{restriction of $\mc{H}$ to $J$}.

Let $$\nm{\mc{H}}=\nm{{H}_1 \times {H}_2 \times \ldots \times {H}_{l}}$$
denote the graph obtained by taking the product of the patches of $\mc{H}$ in order. We say that a patchwork $\mc{H}$ is \emph{linked} (resp. \emph{refined}) if every patch in $\mc{H}$ is linked (resp. refined).  The following useful observation is immediately implied by the properties of linked patches given in the previous subsection.

\begin{obs}\label{o:patchwork1} Let $\mc{H}$ be a linked $q$-patchwork then for every $J \subseteq |\mc{H}|$ we have
$\nm{\mc{H}|_J} \leq \nm{\mc{H}}$.
\end{obs}

We say that  a $q$-patchwork  $\mc{H}=(H_1,H_2,\ldots,H_l)$ is \emph{embedded} in a graph $G$ if
\begin{description}
	\item[(E1)] Every patch in $\mc{H}$ is embedded in $G$;
	\item[(E2)] If $v \in V(H_i) \cap V(H_j)$ for some $i,j \in [l]$, $i \neq j$  then $v \in \dlH_k \cap \drH_k$ for all $k \in [l]$, and we say that such a vertex $v$ is \emph{$\mc{H}$-global}; 
	\item[(E3)] For every $v \in V(\mc{H})$ and every $i \in [l]$ such that $v$ has a neighbor in $V(H_i)$ we have $v \in V(H_i)$;
	\item[(E4)] For every $v \in V(G) - V(\mc{H})$ there exists $i \in [l]$ such that either all the neighbors of $v$ in  $V(\mc{H})$ lie in $\dlH_i$ or they are all in $\drH_i$.	
\end{description}

Let $\mc{H}= (H_i)_{i \in [l]}$ be a  $q$-patchwork  embedded in  a graph $G$,  and let $I=\{i_1,i_2,\ldots,i_j\} \subseteq [l]$, we denote $G \slash H_{i_1} \slash H_{i_2} \ldots \slash H_{i_j}$ by $G \slash (\mc{H},I)$.

\begin{lem}\label{l:contract2prelim} Let $\mc{H}=({H}_1,{H}_2,\ldots,{H}_{l})$ be a  patchwork embedded in a graph $G$. Then the patchwork $\mc{H}|_{[l]-\{i\}}$ is a patchwork embedded in $G \slash H_i$ for every $i \in [l]$ such that $H_i$ is linked.
\end{lem}

\begin{proof}
A routine verification shows that the conditions (E1)--(E4) are preserved by contraction of $H_i$, and that the remaining patches of $\mc{H}$ are unchanged.
\end{proof}

The next corollary is an immediate consequence of Lemmas~\ref{l:contract1} and~\ref{l:contract2prelim}.

\begin{cor}\label{l:contract2} Let $\mc{H}$ be a refined $q$-patchwork embedded in a graph $G$, and  let $l = |\mc{H}|$. Then for every $I \subseteq [l]$, we have $G \slash (\mc{H},I) \leq G$ and
	$$|E(G \slash (\mc{H},I))| = |E(G)|-\sum_{i \in I}e(H_i).$$
\end{cor}

We say that a $q$-patchwork  $\mc{H}=(H_1,\ldots,H_l)$ is \emph{stitched to $G$ by  a collection of paths $\mc{P}=(P_{i,j})_{i \in [q], j \in [l-1]}$} if 	
	\begin{itemize}
		\item $\mc{H}$ is embedded in $G$;
		\item the path $P_{i,j}$ has ends $b_{H_j}(i)$ and $a_{H_{j+1}}(i)$, and $P_{i,j}$ is otherwise disjoint from $V(\mc{H})$;
		\item if $V(P_{i,j}) \cap V(P_{i',j'}) \neq \emptyset$ for some $j \leq j'$ then $i =i'$ and either $j=j'$ or $j=j'-1$. In the second case,   $P_{i,j} \cap P_{i',j'}$ is a common subpath of these two paths with an end $a_{H_{j'}}(i)=b_{H_{j'}}(i)$.
	\end{itemize} 
We say that $\mc{H}$ is \emph{stitched to $G$} if  $\mc{H}$ is stitched to $G$ by some collection of paths.

\begin{lem}\label{l:patchwork2}
	Let $\mc{H}$ be a  patchwork stitched to a graph $G$. Then $\nm{\mc{H}} \leq G$.
\end{lem}

\begin{proof} Let $\mc{P}$ be a collection of paths such that $\mc{H}$ is stitched to $G$ by $\mc{P}$
	The graph $\nm{\mc{H}}$ can be obtained from $G$ by contracting the edges of all paths in $\mc{P}$ and then deleting all the vertices in $V(G) - V(\mc{H})$.
\end{proof}

We say that a patchwork $\mc{H}$ embedded in a graph $G$ is \emph{$S$-respectful} for  $S \subseteq V(G)$  if  every $v \in S \cap V(\mc{H})$ is $\mc{H}$-global.  With all the definitions in place, we finish this section by stating our main structural result, which is proved in Sections~\ref{s:struct} and~\ref{s:large}.

\begin{thm}\label{t:struct2}
	For every proper minor-closed class $\mc{F}$ there exists $\theta=\theta_{\ref{t:struct2}}(\mc{F})$ such that for all integers $l,s \geq 0$ there exists $N=N_{\ref{t:struct2}}(\mc{F},l,s)$ satisfying the following. For every graph $G \in \mc{F}$  with $|V(G)| \geq N$ and every $S \subseteq V(G)$ with $|S| \leq s$,  
	there exists an integer $0 \leq q \leq \theta$ and a refined $S$-respectful $q$-patchwork $\mc{H}$  stitched to $G$ auch that $|\mc{H}| \geq l$ and $|V(H)| \leq N$ for every $H \in \mc{H}$.
\end{thm}

\section{Proof of Theorems~\ref{t:pathwidth} and ~\ref{t:periodic}}\label{s:proofs}

In  this section we derive Theorem~\ref{t:pathwidth}\footnote{and therefore Theorems~\ref{t:rational} and~\ref{t:pathwidth0}} and  Theorem~\ref{t:periodic}, from Theorem~\ref{t:struct2}. For this purpose fix a proper minor-closed class $\mc{F}$ with $d(\mc{F}) >0$, let $\Delta=d(\mc{F})$, and let $\theta=\theta_{\ref{t:struct2}}(\mc{F})$.

For a $q$-patch $H$ let $$\phi(H)=e(H)- \Delta(|V(H)|-q).$$
In particular, $\phi(G) = |E(G)|-\Delta |V(G)|$ for a graph $G$.
We say that a graph $G$ is  \emph{$p$-pruned} if $\phi(G) \geq \phi(H)$ for every minor $H$ of $G$ with  $|V(H)| \geq |V(G)| - p$.

\begin{lem} \label{l:pruned1}
For all $p>0$  there exists $\eps=\eps_{\ref{l:pruned1}}(p)>0$ such that for all $N_0>0$ there exists $N=N_{\ref{l:pruned1}}(p,N_0)$ satisfying the following. For every graph $G$ with $|V(G)| \geq N$ and $d(G) \geq \Delta -\eps$ there exists a graph $G'$ with $|V(G')| \geq N_0$ such that $G' \leq G$  and either \begin{description}
	\item[(P1)] $G'$ is $p$-pruned, and $\phi(G') \geq \phi(G)$, or
	\item[(P2)]  $d(G') \geq \Delta +\eps$.
	\end{description}
\end{lem}

\begin{proof}
	Let $$\delta= \min_{k \in [p]} \{1 + \Delta k - \lceil \Delta k \rceil\}.$$
	Then $\delta > 0$.
	Note that if a graph $G$ is not $p$-pruned then there exists a minor $H$ of $G$ such that $|V(H)| \geq |V(G)| - p$, and
	$$0 < \phi(H) -\phi(G) = \Delta(|V(G)|-|V(H)|) - |E(G)|-|E(H)|.$$ 
	Thus $|E(G)|-|E(H)| <  \lceil \Delta(|V(G)|-|V(H)|) \rceil$, and so $\phi(H) \geq \phi(G) + \delta$.
	
	We show that $\eps = \frac{\delta}{2p}$ and $N =  3(N_0+p)$ satisfy the lemma.
	Let a minor $G'$ of $G$ chosen so that $|V(G')| \geq N_0$,
	\begin{equation}\label{e:pruned1}
	\phi(G') \geq \phi(G) + 2\eps(|V(G)|-|V(G')|),
	\end{equation} and subject to these condition $|V(G')|$ is minimum. 
	Note that such a choice is possible as $G'=G$ satisfies \eqref{e:pruned1}.
	 
	Suppose first that $|V(G')| < N_0 + p$ then
	\begin{align*}
	|E(G')| &\stackrel{\eqref{e:pruned1}}{\geq} \Delta |V(G')|+ (|E(G)| - \Delta|V(G)|) +  2\eps(|V(G)|-|V(G')|) \\ &\geq (\Delta +\eps) |V(G')| + \eps |V(G)| -3\eps|V(G')|
	 \\ &\geq (\Delta +\eps) |V(G')|,
	\end{align*}
	where the second inequality holds as $|E(G)|\geq (\delta-\eps)|V(G)|$, and the last inequality holds by the choice of $N$. It follows that (P2) holds in this case.
	
	Thus we assume that $|V(G')| \geq  N_0 + p$. If $G'$ is $p$-pruned then (P1) holds. Otherwise, there exists a minor $H$ of $G'$ such that $|V(H)| \geq |V(G)| - p \geq N_0$, and $$\phi(H) \geq \phi(G') + \delta \geq  \phi(G) + 2\eps(|V(G)|-|V(G')|+p) \geq \phi(G) + 2\eps(|V(G)|-|V(H)|),$$ 
	where the second inequality holds by the choice of $\eps$.
	Thus $H$ contradicts the choice of $G'$.
\end{proof}

\begin{cor} \label{l:pruned} 
For every $N>0$ there exists $G \in \mc{F}$ such that $|V(G)| \geq N$, 
 and $G$ is $N$-pruned. 
\end{cor}

\begin{proof}
Let $\eps=\eps_{\ref{l:pruned1}}(N)$. Let $N_0 \geq N$ be chosen so that $d(G') < \Delta +\eps$ for every $G' \in \mc{F}$ with $|V(G')| \geq N_0$, and let $N_1=N_{\ref{l:pruned1}}(N,N_0)$. By the choice of $\Delta$, there exists $G \in \mc{F}$  with $|V(G)| \geq N_1$ and 	$d(G) \geq \Delta - \eps$. By the choice of $N_1$ there exists a minor $G' \in \mc{F}$ of $G$ such that $|V(G')| \geq N_0 \geq N$, and either 
$G'$ is $N$-pruned, or $d(G') \geq \Delta +\eps$. In the first case, $G'$ satisfies the corollary, while the second case is impossible by the choice of $N_0$.
\end{proof}

 We say that a $q$-patch $H$ is \emph{heavy} if $\phi(H) > 0$, $H$ is \emph{balanced} if $\phi(H) = 0$, and $H$ is  \emph{light} if $\phi(H) < 0$.

\begin{cor}\label{c:notheavy}
	Let $H$ be a refined $\mc{F}$-tame $q$-patch. Then $H$ is not heavy.
\end{cor}

\begin{proof}
	Using \cref{c:limitdensity} we have
	 $$\Delta \geq \lim_{n \to \infty}d(\nm{H^n}) \geq \frac{e(H)}{|V(H)|-q} = \Delta + \frac{\phi(H)}{|V(H)|-q}.$$
	Thus $\phi(H) \leq 0$ as desired.
\end{proof}

\begin{lem} \label{l:heavy} 
	There exists $K=K_{\ref{l:heavy}}$ such that for every  $q \leq \theta$, every refined $q$-patchwork $\mc{H}$ such that $\llbracket \mc{H} \rrbracket \in \mc{F}$ contains at most $K$ heavy patches.
\end{lem}

\begin{proof}
	Clearly, it suffices to show that there exists $K$ satisfying the lemma requirements for any fixed $q \leq \theta$.
	By \cref{t:wqo} there exists a finite set $\mc{M}$ of refined  heavy $q$-patches such that for every refined heavy $q$-patch $H$ there exists $M \in \mc{M}$  such that $M \leq H$.  Let $M = |\mc{M}|$, and let $k=K_{\ref{l:multiply}}(\mc{F},q)$.
	
	We show that $K=kM$ is as desired. 
	Suppose for a contradiction that a refined patchwork $\mc{H}$ contains at least $K$ heavy $q$-patches. Then by the pigeonhole principle there exists $M \in \mc{M}$ such that $M$ is a minor of at least $k$ elements of $\mc{H}$. Thus $\nm{M^k} \leq \nm{\mc{H}}$ by \cref{o:patchwork1}, and so $\nm{M^k} \in \mc{F}$.  By the choice of $k$ it follows that $M$ is $\mc{F}$-tame, contradicting \cref{c:notheavy}. 
\end{proof}

\begin{lem} \label{l:faddition}
Let $\mc{H}=({H}_1,\ldots,{H}_{l})$ be a  $q$-patchwork  embedded in a graph $G$, then for every $I \subseteq [l]$
 \begin{equation}
 \label{e:faddition} \phi(G \slash (\mc{H},I)) = \phi(G)-\sum_{i \in I}\phi(H_i).\end{equation}
\end{lem}

\begin{proof}
By \cref{l:contract2} we have \begin{equation}
\label{e:faddition1} |E(G)|-|E(G \slash (\mc{H},I)|= \sum_{i \in I}e(H_i)
\end{equation}	and clearly \begin{equation}
\label{e:faddition2}
|V(G)|-|V(G \slash (\mc{H},I)| = \sum_{i \in I}(|V(H_i)|-q).\end{equation} Multiplying \eqref{e:faddition2} by $\Delta$ and subtracting the result from \eqref{e:faddition1}, we obtain \eqref{e:faddition}.
\end{proof}

We are now ready to derive the first of our main results from \cref{t:struct2}.

\begin{proof}[Proof of Theorem~\ref{t:pathwidth}]
	By \cref{t:wqo} there exists a finite set $\mc{M}$ of refined balanced $q$-patches such that for every refined balanced $q$-patch $H$ there exists $M \in \mc{M}$  such that $M \leq H$. 
	Let $K_0 = \max_{q \leq \theta}K_{\ref{l:multiply}}(\mc{F},q)$, let $l=|\mc{M}|K_0 + K_{\ref{l:heavy}}$, and let  
	$N=N_{\ref{t:struct2}}(\mc{F},l,0)$. 
	
	By Lemma~\ref{l:pruned} there exists an $N$-pruned $G \in \mc{F}$ with $|V(G)| \geq N$. By the choice of $N$,   there exists $q \leq \theta$ and a refined $q$-patchwork  $\mc{H}=(H_i)_{i \in [l]}$ stitched to $G$  such that $|V(H_i)| \leq N$ for every $i \in [l]$. 
	
	Suppose that $H_i$ is light for some $i \in [l]$. Then $\phi(G \slash (\mc{H},\{i\})) > \phi(G)$ 
	by Lemma~\ref{l:faddition} and $G \slash (\mc{H},\{i\}) \leq G$ by \cref{l:contract2}. As $|V(G \slash (\mc{H},\{i\}))| \geq |V(G)|-N$, it follows that $G$ is not $N$-pruned, a contradiction. Thus no element of $\mc{H}$ is light, while 
	by Lemma~\ref{l:heavy} at most  $ K_{\ref{l:heavy}}$ elements of $\mc{H}$ are heavy. 
	
	Thus there exists $J \subseteq [l]$ such that $|J| \geq |\mc{M}|K_0$ and $H_i$ is balanced for every $i \in J$. By the pigeonhole principle there exists a refined balanced patch $M \in \mc{M}$ and $I \subseteq J$ with $|I|=K_0$ such that $M \leq H_i$ for every $i \in I$. Thus by \cref{o:patchwork1} and \cref{l:patchwork2} we have $\nm{M^{K_0}} \leq \llbracket \mc{H} \rrbracket \leq G$. By the choice of $K_0$, it follows that $M$ is $\mc{F}$-tame. By \cref{c:limitdensity} we have $\lim_{n \to \infty} d(M^n) \geq \Delta$, and so  $\lim_{n \to \infty} d(M^n) = \Delta$, as desired.
\end{proof}

It remains to prove Theorem~\ref{t:periodic}.
Let $f(n)=ex_{\F}(n)-\Delta n$. Thus Theorem~\ref{t:periodic} states that $f(n)$ is eventually periodic. 
We say that a graph $G \in \mc{F}$ is \emph{extremal} if $|E(G)|=ex_{\F}(|V(G)|)$, or, equivalently, $\phi(G)=f(|V(G)|)$.

\begin{lem}\label{c:bounded} There exists $D=D_{\ref{c:bounded}}$ such that $|f(n)| \leq D$ for all $n$.
\end{lem}

\begin{proof}
First we show that $f(n)$ is bounded from below.
Let $q$ and $H$ be as in Theorem~\ref{t:pathwidth}. Without loss of generality we assume that $E(\nm{H}|[\dlH])= \emptyset.$ Let $h=|V(H)|-q$. Thus $|V(H^k)|=kh+q$ and $|E(H^k)|=k|E(H)|$, and so $|E(H)|=\Delta h$.
As $ex_{\F}(n) \geq |E(H^{\lfloor(n-q)/h\rfloor})|$, we have $ex_{\F}(n) \geq \Delta h ((n-q)/h-1)$, implying $f(n) \geq -\Delta(q+h)$, as desired.

It remains to show that $f(n)$ is bounded from above.
Let	$N=N_{\ref{t:struct2}}(\mc{F},K_{\ref{l:heavy}}+1,0)$. Thus for every $G \in \mc{F}$ with $|V(G)| \geq N$ there exists a refined $q$-patchwork $\mc{H}$ stitched to $G$ for some $q \leq \theta$ with $|\mc{H}|> K_{\ref{l:heavy}}$. In particular,  there exists a patch $H \in \mc{H}$ such that $\phi(H) \leq 0$.

We will show that  $f(n) \leq \max_{n' < N} f(n')$ for every positive integer $n$, which will finish the proof of the lemma.
Suppose not. Then there exists $n \geq N$ such that $f(n) > f(n')$ for every $n' < n$. Let $G \in \mc{F}$ with $|V(G)|=n$ be extremal. Let $H$ be a  patch embedded in $G$, as in the previous paragraph. By Lemma~\ref{l:faddition} we have $\phi(G \slash H) \geq \phi(G)$. Thus $f(|V(G \slash H)|) \geq f(n)$, contradicting the choice of $n$.
\end{proof}

\begin{lem}\label{l:balanced1} For all $s,c,l$ there exists $N=N_{\ref{l:balanced1}}(s,c,l)$ such that for every $G \in \mc{F}$ with $|V(G)| \geq N$ and $\phi(G) \geq c$ and every $S \subseteq V(G)$ with $|S| \leq s$ there exists an $S$-respectful $q$-patchwork $\mc{H}$ embedded in $G$  such that $|\mc{H}| \geq l$,   $|H| \leq N$ for every $H \in \mc{H}$ and every patch in $\mc{H}$ is balanced.  
\end{lem}

\begin{proof}
By Theorem~\ref{t:rational}, there exists $\eps >0$ such that $\phi(H) \leq - \eps$ for every light patch $H$.  Let $l' = l + K_{\ref{l:heavy}} + (D_{\ref{c:bounded}}-c)/\eps$. We show that $N=N_{\ref{t:struct2}}(l',s)$ satisfies the lemma.

Let $G,S$ be as in the lemma statement. By the choice of $N$ there exists an $S$-respectful $q$-patchwork $\mc{H}$ stitched to $G$ such that $q \leq \theta$, $|\mc{H}| \geq l'$  and $|H| \leq N$ for every $H \in \mc{H}$. Suppose for a contradiction that $\mc{H}$ contains fewer than $l$ balanced patches. As  $\nm{\mc{H}} \leq G$ by \cref{l:patchwork2},  we have $ \nm{\mc{H}} \in \mc{F}$, and so $ \nm{\mc{H}}$ contains at most $K_{\ref{l:heavy}}$ heavy patches. It follows that $\mc{H}$ contains more than $(D_{\ref{c:bounded}}-c)/\eps$ light patches. Let $I$ be the set of indices of light patches in $G$. By Lemma~\ref{l:faddition}  we have $\phi(G \slash(\mc{H},I)) \geq \phi(G)+|I|\eps > D_{\ref{c:bounded}}$, a contradiction. 	
\end{proof}

\begin{lem}\label{l:balanced2} There exists $p=p_{\ref{l:balanced2}}$ such that for every $l$ there exists $N=N_{\ref{l:balanced2}}(l)$ such that for every extremal $G \in \mc{F}$ with $|V(G)| \geq N$ there exists $q \leq \theta$ and a $q$-patchwork $\mc{H}$ embedded in $G$ such that $|\mc{H}|\geq l$, every $H \in \mc{H}$ is balanced and satisfies $|V(H)| \leq p$.
\end{lem}

\begin{proof}
	We show that $p=N_{\ref{l:balanced1}}(2\theta,-\theta\Delta ,1)$ and $N=N_{\ref{l:balanced1}}(0,-D_{\ref{c:bounded}}, (\theta+1) 2^{\theta} l)$ satisfy the lemma.  Let $G \in \mc{F}$ with $|V(G)| \geq N$ be extremal. By the choice of $N$ there exists $q' \leq \theta$, $L \geq \theta 2^{\theta} l$ and a $q'$-patchwork $\mc{H}'=(H'_i)_{i \in [L]}$ embedded in $G$ such that  $H'_i$ is balanced for every $i \in [L]$. Let $A'$ be the set of $\mc{H}'$-global vertices.
	
	Note that $\phi(G[V(H'_i)]) \geq -\theta\Delta$ for every $i \in [L]$, as $H'_i$ is embedded in $G$, and $\phi(H'_i)=0$.

	By the choice of $p$, applying \cref{l:balanced1} to $G([V(H'_i)])$ for every $i \in [L]$ such that $|V(H'_i)| \geq p$, there exists $q_i \leq \theta$ and a $\dH'_i$-respectful balanced $q_i$-patch $H_i$ embedded in $G[V(H'_i)]$ such that $|V(H_i)| \leq p$. It is easy to see that such a patch $H_i$ is embedded in $G$. 	
	
	If $|V(H'_i)|< p$,  we define $H_i = H'_i$. 
	
	By the pigeonhole principle there exists $q \leq \theta,$ $A \subseteq A'$ and $I \subseteq [L]$ with $|I|=l$ such that $q_i = q,$ and $V(H_i) \cap A' = A$ for every $i \in I$. Let $\mc{H}=(H_i)_{i \in I}$. A routine verification shows that $\mc{H}$ is a $q$-patchwork embedded in $G$. (In particular, $A$ is the set of all $\mc{H}$-global vertices). Thus $\mc{H}$ satisfies the lemma. 
\end{proof}

We precede the final steps in the proof of \cref{t:periodic} by a simple arithmetical lemma.
	
\begin{lem}\label{l:add}
	Let $p$ be a positive integer, let $m=p \cdot p!$, and let $c_1,c_2, \ldots, c_{m}$ be integers such that $c_i \in [p]$ for every $i \in [m]$. Then there exists $I \subseteq [m]$ such that $\sum_{i \in I}c_i = p!$
\end{lem}
\begin{proof}
	By the pigeonhole principle we may assume that $c_1 = c_2 = \ldots = c_{p!}$. Then $I=[p!/c_1]$ satisfies the lemma.
\end{proof}

\begin{cor}\label{c:periodic1} There exists positive integers $P=P_{\ref{c:periodic1}}, N=N_{\ref{c:periodic1}}$ such that $f(n-P) \geq f(n)$ for every $n \geq N$.
\end{cor}

\begin{proof}
	Let $p=p_{\ref{l:balanced2}}$, and let $m=p \cdot p!$. We show that the corollary holds for $P=p!$, and $N=N_{\ref{l:balanced2}}(m)$. 
	
	For $n \geq N$, let $G \in \mc{F}$ be an extremal graph with $|V(G)|=n$.  By the choice of $N$ there exists $q \leq \theta$ and a $q$-patchwork $\mc{H}=(H_i)_{i \in [m]}$ embedded in $G$ such that $H_i$ is balanced and satisfies $|V(H_i)| \leq p$ for every $i \in [m]$. By \cref{l:add} applied to the sequence $\{|V(H_i)|-q\}_{i \in [m]}$ there exists $I \subseteq [m]$  such that $\sum_{i \in I} (|V(H_i)|-q) = P$. Using \cref{l:faddition} we have $$f(n-P) \geq  \phi(G \slash (\mc{H},I)) \geq \phi(G) = f(n),$$ as desired.
\end{proof}

\begin{proof}[Proof of Theorem~\ref{t:periodic}] Let $P=P_{\ref{c:periodic1}}$. By Theorem~\ref{t:rational} there exists an integer $m$ such that $mf(n)$ is integral for every $n$, and so by \cref{c:bounded} the set of values of $f(n)$ is finite.

 For every $i \in [P]$, let $a_i = \lim \inf_{n \to \infty} f(nP+i)$. 
  By the observation above, there exists $M_i$ so that $f(M_iP + i) = a_i$, $M_i \geq N_{\ref{c:periodic1}},$ and $f(nP+i) \geq a_i$ for every $n \geq M_i$. It follows from Corollary~\ref{c:periodic1} that $f(nP+i)=a_i$ for every $n \geq M_i$, implying that $f(n)$ is eventually periodic with period $P$, as desired.
\end{proof}

\section{Proof of Theorem~\ref{t:struct2}}\label{s:struct}

In this section we take  the first steps in the proof of  Theorem~\ref{t:struct2} reducing it to the large treewidth case, more specifically to the case when $G$ contains moderately wide and very long grid minor, which is handled in Theorem~\ref{t:large}, proved in the next session.

We start by introducing the necessary terminology related to tree decompositions,  starting with the standard definitions. 
A \emph{tree decomposition} of a graph $G$ is a pair $(T, \beta)$, where $T$ is a tree and  $\beta$ is a function that to each vertex $t$ of $T$ assigns a subset of
vertices of $G$ (called the \emph{bag} of $t$), such that for every $uv\in E(G)$, there exists $t\in V(H)$ with $\{u,v\}\subseteq \beta(t)$, and
for every $v\in V(G)$, the set $\{t:v\in\beta(t)\}$ induces a non-empty connected subgraph of $H$.
Given a tree decomposition $(T,\beta)$ for an edge $e=st \in E(T)$, let $\beta(e):=\beta(s) \cap \beta(t)$. 
The \emph{order} of a tree decomposition  $(T, \beta)$ is $|V(T)|$, its \emph{width} is equal to $\max_{v \in V(T)}|\beta(v)|$, and its  \emph{adhesion}  is equal to $\max_{e \in E(T)}|\beta(e)|$. The \emph{treewidth} $\tw(G)$ of a graph $G$  is equal to the minimum width of a tree decomposition of $G$, and  the \emph{$\theta$-treewidth} $\tw_{\theta}(G)$  is equal to the minimum width of a tree decomposition of $G$ of adhesion less than $\theta$.
 
A tree decomposition $(T, \beta)$  is \emph{a path decomposition} if $T$ is a path.
We say that a tree decomposition  $(T, \beta)$ is \emph{perfectly linked} if for all edges $e_1,e_2$ of $T$ there exists a $(\beta(e_1),\beta(e_2))$-linkage in $G$ of order $|\beta(e_1)|$. (Thus in particular, $|\beta(e)|$ is the same for all $e \in E(T)$.)  We say that a tree decomposition is \emph{proper} if $\beta(t_1) \not\subseteq\beta(t_2)$ for every $t_1t_2\in E(T)$. A path decomposition of $G$ is \emph{appearance-universal} if every vertex $v\in V(G)$ either appears in all bags of the decomposition,
or in at most two (consecutive) bags.  

The following two lemmas, largely adapted  from ~\cite{DvoNor17}, are used to prove \cref{t:struct2} for graphs with bounded $\theta$-treewidth. For completeness we provide the proof of the first of these lemmas, which is essentially identical to the proof in~\cite[Lemma 25]{DvoNor17}, as its statement is somewhat different. 

\begin{lem}\label{l:tree2path}
	For all $l,w$ there exists $N=N_{\ref{l:tree2path}}(l,w)$ satisfying the following. Let $G$ be a graph with $|V(G)| \geq N$, and let $\theta$ be such that $tw_{\theta}(G) \leq w$, then
	$G$ has a proper path decomposition of adhesion at most $\theta$ and order at least $l$.
\end{lem}

\begin{proof}We show that $N = w((l-1)2^w)^l$ satisfies the lemma. Let $(T,\beta)$ be a proper tree decomposition of $G$ of width at most $w$ and adhesion at most $\theta$. Then $|V(T)| \geq N/w = ((l-1)2^w)^l$.
	
Suppose first that $T$ has a vertex $z$ of degree $m \ge (l-1)2^w$, and let $T_1$, \ldots, $T_m$ be the components of $T-z$, and let $R_i= \beta(z) \cap \bigcup_{x\in V(T_i)} \beta(x)$. By the pigeonhole principle we may assume that $R_1=R_2=\ldots=R_{l-1}$. Let $H$ be a path with $V(H)=[l]$, and let $\gamma(i)=\bigcup_{x\in V(T_i)} \beta(x)$ for $i \in [l-1]$, and let $\gamma(l)= \bigcup_{x\in V(T)-\cup_{i \in [l-1]} V(T_i)} \beta(x)$. 

Otherwise, $T$ contains a subpath $H$ with at least $l$ vertices. For $z\in V(H)$,
let $T_z$ be the component of $T-(V(H)\setminus\{z\})$ containing $z$, and let $\gamma(z)=\bigcup_{x\in V(T_z)} \beta(x)$.
In both cases, $(H,\gamma)$ is a proper path decomposition of $G$ of adhesion at most $\theta$.
\end{proof}

The next lemma is a combination of two lemmas from~\cite{DvoNor17}.

\begin{lem}[\protect{\cite[Lemmas 26 and 27]{DvoNor17}}]\label{l:goodpath}
	For all $l,\theta$  there exists $L=L_{\ref{l:goodpath}}(l, \theta)$ such that if a graph $G$ admits a proper path decomposition of adhesion at most $\theta$ and order at least $L$ then $G$ admits a perfectly-linked, appearance universal, proper path decomposition of adhesion at most $\theta$ and order at least $l$. 
\end{lem}

Lemmas~\ref{l:tree2path} and~\ref{l:goodpath} allow us to prove a version of \cref{t:struct2} in the case when $\theta$-treewidth of $G$ is very small compared to $|V(G)|$.

\begin{lem}\label{l:smallw}
	For all  $l$, $\theta$, $w$  there exist $N=N_{\ref{l:smallw}}(l, \theta,w) $ satisfying the following. Let $G$ is a graph with $\tw_{\theta}(G) \leq w$, $|V(G)| \geq N$ then for some $q \leq \theta$
	there exists a refined $q$-patchwork $\mc{H}$ with $|\mc{H}| \geq l$ stitched to $G$.
\end{lem}

\begin{proof}
	Let $k=3l$, and let $N=N_{\ref{l:tree2path}}(L_{\ref{l:goodpath}}(k, \theta),w)$. By the choice of $N$, the graph $G$ admits a perfectly-linked, appearance universal, proper path decomposition $(R,\beta)$ for some path $R$ with $|V(R)| = k$. Let $u_1,\ldots,u_k$ be the vertices of $R$ in order. Let $e_i = u_iu_{i+1}$. As  $(R,\beta)$ is perfectly-linked there exists $q \leq \theta$ such that $|\beta(e_i)|=q$  for every $i \in [k-1]$.  Let $\mc{Q}=(Q_1,\ldots,Q_q)$ be a $(\beta(e_1),\beta(e_{k-1}))$-linkage in $G$. For all $i \in [q], j \in [k-1]$, let $v_{i,j}$ denote the unique vertex in $V(Q_i) \cap \beta(e_j)$. 
		
	Let $H_j$ be a $q$-patch with $\nm{H_j}= G[\beta(u_{j+1})]$, $a_{H_j}(i)=u_{i,j}, b_{H_j}(i)=u_{i,j+1}$. Then $H_j$ is embedded in $G$. The patch $H_j$  is  non-degenerate, as  $(R,\beta)$ is proper, and the restriction of $\mc{Q}$ to $H_j$ shows that $H_j$ is linked. Thus, $H_j$ is refined.
	
	We claim that the refined $q$-patchwork $\mc{H}=(H_{3j+1})^{l-1}_{j=0}$ satisfies the lemma.  
	
	First, we verify that $\mc{H}$ is embedded in $G$.  The condition (E1) has been verified above.  The condition (E2) holds as $(R,\beta)$ is appearance-universal. In particular, $\mc{H}$-global vertices are exactly the vertices of $G$ appearing in all the bags of $(R,\beta)$. We say that these vertices are \emph{global} and the remaining vertices of $G$ are  \emph{local}.
	
	If $v \in V(G)$ has a local neighbor in $V(H_{3j+1})= \beta(u_{3j+2})$, then $v \in  \beta(u_{3j+1}) \cup \beta(u_{3j+2}) \cup \beta(u_{3j+3})$. Suppose first that $v \in V(H_{3j+1})$ is local. Then  all the neighbors of $v$ in $V(\mc{H})$ lie in $V(H_{3j+1})$. As (E3) clearly holds for every global vertex, it follows that (E3) holds for every $v \in V(\mc{H})$. 
	
	If $v \in V(G) - V(\mc{H})$ then either \begin{itemize}
		\item $v \in \beta(u_{3j+1}) - \beta(u_{3j+2})$, in which case all the neighbors of $v \in V(\mc{H})$ lie in $\dlH_{3j+1}$, or
			\item $v \in \beta(u_{3j+3}) - \beta(u_{3j+2})$, and all the neighbors of $v \in V(\mc{H})$ lie in $\drH_{3j+1}$.
	\end{itemize}
	Thus (E4) holds, and so $\mc{H}$ is embedded in $G$. 
	
	Finally,	let $\mc{P}=(P_{i,j})_{i \in [q], j \in [l-1]}$ be a collection of paths, where $P_{i,j}$ is the subpath of $Q_i$ with ends $v_{i,3j+2}=b_{H_{3j+1}}(i)$ and $v_{i,3j+4}=a_{H_{3j+4}}(i)$. It is routine to verify that $\mc{H}$  is stitched to $G$ by $\mc{P}$, as claimed.
\end{proof}

In the remaining case, when $\theta$-treewidth is large, we utilize the results of  Geelen and Joeris~\cite{GeeJoe16}, which we restate using the terminology of this paper.

Let  $T$ be a tree, and let $Z \subseteq V(T)$ be non-empty, such that every $v \in V(T)-Z$ is a leaf of $T$. We assume for convenience that $V(T)=[q]$ and $Z=[z]$.  We define a \emph{$Z$-extension of $T$} to be a $q$-patch $T'$ defined as follows:
\begin{itemize}
	\item $V(T')=[q + z]$, 
	\item $E(T')=E(T) \cup \{\{i,(i+q)\}\}_{i \in Z},$
	\item $a_{T'}(i)=i$ and $b_{T'}(i)=i+q$ for $i \in Z$, and
	\item $a_{T'}(i)=b_{T'}(i)=i$ for $i \in V(T)-Z$.
\end{itemize}	
Thus $T'$ is a refined patch, which we call a \emph{tree-extension patch}. 

The main result of~\cite{GeeJoe16} implies that a large enough  $\theta$-treewidth implies existence of any fixed power of a tree-extension patch on $\theta$ vertices or a big complete unbalanced bipartite minor.

\begin{thm}[\protect{\cite[Theorem 1.1]{GeeJoe16}}]\label{t:wheel}
	For every pair of integers $\theta \geq 2, n \geq 1$  there exists $w = w_{\ref{t:wheel}}(\theta,n)$, such that for every graph $G$ with $\tw_{\theta}(G) \geq w$ either $K_{\theta,n} \leq G$, or $\nm{T^n} \leq G$ for some tree-extension patch $T$ with $|V(T)| \geq \theta$.
\end{thm}

Geelen and Joeris~\cite{GeeJoe16} also note that powers of tree-extension patches can be converted to grid minors, which for us are more convenient to work with.
An \emph{$(n,s)$-grid} is a graph with the vertex set $[n] \times [s]$ and an edge joining vertices $(x_1,y_1)$ and $(x_2,y_2)$ if and only if $|x_1-x_2|+|y_1-y_2|=1$. That is an $(n,s)$-grid is a Cartesian product of a path on $s$ vertices and a path on $n$ vertices.
The following lemma, implicit in the proof of \cite[Theorem 10.1]{GeeJoe16}.

\begin{lem}[\protect{\cite{GeeJoe16}}]W\label{l:treepower} For every positive integer $s$ there exists $\theta=\theta_{\ref{l:treepower}}(s)$ such that for every tree-extension patch $T$ with $|V(T)| \geq \theta$ and every positive integer $n$ the graph $\nm{T^n}$ contains either $K_{s,n}$ or an $(n,s)$-grid as a minor. 
\end{lem}

The final and most technical ingredient of the proof of \cref{t:struct2} is the following theorem, which combined with \cref{t:wheel} and \cref{l:treepower} guarantees a large refined $q$-patchwork in the ``large treewidth'' case.

\begin{thm}[\protect{\cite{GeeJoe16}}]\label{t:large}
	For every integer $t \geq 5$ there exists $s=s_{\ref{t:large}}(t)$ such that for every $l$  there exist $n=n_{\ref{t:large}}(t,l)$ satisfying the following. Let $G$ be a graph  with an $(n,s)$-grid minor such that $K_t \not \leq G$. Then for some $q \leq t-2$ there exists a refined $q$-patchwork $\mc{H}$  with $|\mc{H}| \geq l$ stitched to $G$.
\end{thm}

The proof of \cref{t:large} is postponed to \cref{s:large}.

Combining the above results we can now prove the following weaker version of \cref{t:struct2}.

\begin{lem}\label{l:struct}
	For every proper minor-closed class $\mc{F}$ there exists $\theta=\theta_{\ref{l:struct}}(\mc{F})$ such that for every $l$  there exist $N=N_{\ref{l:struct}}(\mc{F},l)$ satisfying the following. Let $G \in \mc{F}$ be a graph with $|V(G)| \geq N$ then for some $q \leq \theta$
	there exists a refined $q$-patchwork $\mc{H}$  with $|\mc{H}| \geq l$ stitched to $G$.
\end{lem}

\begin{proof}
	Let $t \geq 5$ be such that $K_t \not \in \mc{F}$. Let  $s=\max\{s_{\ref{t:large}}(t),t\}$,  $n=\max\{n_{\ref{t:large}}(t,l),t\}$,  $\theta=\max\{\theta_{\ref{l:treepower}}(s), t\}$, and $w= w_{\ref{t:wheel}}(\theta,n)$. Finally, let $N=N_{\ref{l:smallw}}(l, \theta, w)$. 
	
	Let $G \in \mc{F}$ be a graph with $|V(G)| \geq N$. Suppose first that $\tw_{\theta}(G) \geq w$. Then by \cref{t:wheel} either $K_{\theta,n} \leq G$ or $\nm{T^n} \leq G$ for some tree-extension patch $T$ with $|V(T)| \geq \theta$. In the first case we have $K_t \leq K_{t,t} \leq G$, contradicting the choice of $t$. In the second case,  using \cref{l:treepower} we conclude that $G$ has an $(n,s)$-grid minor, and so the lemma holds by \cref{t:large}.
	
	In the remaining case $\tw_{\theta}(G) \leq w$,  and so \cref{l:smallw} applies to $G$ by the choice of $N$, and so our lemma also holds.
\end{proof}

Finally, we bootstrap \cref{l:struct} to obtain  \cref{t:struct2}. We precede the proof by two useful observations, which are immediate from the definitions.

\begin{obs}\label{o:respect1}
	Let $G$ be a graph, let $S \subseteq V(G)$, and let $\mc{H}$  be a patchwork embedded in $G$ with $|\mc{H}| > |S|$. Then there exists an $S$-respectful restriction $\mc{H}'$ of $\mc{H}$ such that $|\mc{H}'| \geq |\mc{H}|-|S|$. 
\end{obs}

\begin{obs}\label{o:respect2}
	Let $(A,B)$ be a separation of a graph $G$, let $S \subseteq V(G)$ be such that $S \cap A \subseteq  A \cap B$. Let $\mc{H}$ be an  $(A \cap B)$-respectful patchwork embedded in $G[A]$. Then $\mc{H}$ is an $S$-respectful patchwork embedded in $G$. 
\end{obs}

\begin{proof}[Proof of \cref{t:struct2}.] Let $\theta = \theta_{\ref{l:struct}}(\mc{F})$. Assume without loss of generality that $l \geq 2$ and $s \geq 2\theta$. Let $k=l+s$, and let $N = N_{\ref{l:struct}}(\mc{F},k)$. 
	
	Suppose for a contradiction that there exists a graph $G \in \mc{F}$ with  $|V(G)| \geq N$ and $S \subseteq V(G)$ with $|S| \leq s$  that do not satisfy the conclusion of the theorem. We choose such a pair $(G,S)$ with $|V(G)|$ minimum. By the choice of $\theta$ and $n$  for some $q_0 \leq \theta$
	there exists a refined $q_0$-patchwork $\mc{H}_0$  with $|\mc{H}_0| \geq k$ stitched to $G$. By \cref{o:respect1}, there exists an $S$-respectful restriction $\mc{H}'$ of $\mc{H}_0$ with $|\mc{H}'| \geq l$. By the choice of $(G,S)$, there exists $H' \in \mc{H}'$ such that $| V(H')| > N$. Let $A = V(H')$ and $B=(V(G)-V(H')) \cup \dH$. Then $(A,B)$ is a separation of $G$, as $H$ is embedded in $G$. Let $S'=A \cap B$. By the choice of $G$ for some $q \leq \theta$ there exists an $S'$-respectful refined  $q$-patchwork $\mc{H}$ stitched to $G[A]$  such that $|V(H)| \leq N$ for every $H \in \mc{H}$. By \cref{o:respect2}, the patchwork $\mc{H}$ is an $S$-respectful patchwork embedded  in $G$. Thus $(G,S)$ satisfies the conclusion of the theorem, a contradiction.
\end{proof}

\section{Proof of \cref{t:large}}\label{s:large}

In this section we prove \cref{t:large} thus completing the proof of \cref{t:struct2} and, in turn, the proofs of Theorems~\ref{t:rational}--\ref{t:periodic}.

Our proof relies on a technical version of the Flat Wall theorem. We start by introducing the definitions necessary to state and subsequently use this result.

The infinite graph $\bb{W}$  has vertex set $\bb{Z}^2$ and an edge between any vertices $(x,y)$ and $(x',y')$ if either \begin{itemize}
	\item $|x-x'|=1$ and $y=y'$, or 
	\item  $x=x'$, $|y-y'|=1$ and $x$ and $\max\{y,y'\}$ have the same parity,
\end{itemize} 
Let $h,l \geq 1$ and $x_0,y_0$ be integers, and let $X=[x_0+1,x_0+2l], Y=[y_0+1,y_0+h]$.   An \emph{$(l,h)$-wall} (or, simply a \emph{wall}) $W=W(X,Y)$ is obtained from the subgraph of $\bb{W}$ induced by $X \times Y$ 
by deleting the vertices of degree one.  Note that $W$ is a subgraph of an $(2l,h)$-grid and the maximum degree of $W$ is three. 
The \emph{outer cycle} of $W$ is the unique cycle $C \subseteq W$ such that $W$ is contained within the disc bounded by $C$ in the  straight line drawing of $\bb{W}$ in the plane. 

The outer cycle $C$ of $W$ can be naturally partitioned into four paths, bottom, top, left and right, meeting at four corners of $W$. Formally, we define these as follows. The path in $C$ induced in it by the vertices with the $y$-coordinate $y_0+1$ (resp. $y_0+h$) is the \emph{bottom} (resp.  \emph{top}) of $W$. We denote these paths by $B(W)$ (resp. $T(W)$). Let $c_{--}(W)$ and $c_{+-}(W)$ denote the ends of $B(W)$, where $c_{--}(W)$ has the smaller $x$-coordinate, and let $c_{-+}(W)$ and  $c_{++}(W)$ denote the ends of $T(W)$ in the same order. We say that $c_{--}(W),c_{-+}(W),c_{+-}(W)$ and $c_{++}(W)$  are  the \emph{corners} of $W$, and that $\{c_{--}(W),c_{++}(W)\}$ and $\{c_{-+}(W),c_{+-}(W)\}$ are the pairs of \emph{opposite corners of $W$}.
The cycle $C$  contains a unique path with ends  $c_{--}(W)$ and $c_{-+}(W)$ (resp.  $c_{+-}(W)$ and $c_{++}(W)$), internally disjoint from the top and bottom, which we call the \emph{left}  (resp. the  \emph{right}) of $W$, and denote by $L(W)$ (resp. $R(W)$).

In this section we study subdivisions of walls appearing as subgraphs within larger graphs. The next definition of embedding helps us track the correspondence between the vertices original wall and the vertices of the subdivision. We also extensively use this definition in our discussion of topological minors in Section~\ref{s:top}.

An \emph{embedding} $\eta$ of a graph $H$ into a graph $G$ is a map with the domain $V(H) \cup E(H)$ such that \begin{itemize}
	\item $\eta(v) \in V(G)$ for every $v \in V(H)$, and $\eta(u)\neq \eta(v)$ for $u \neq v$; \item $\eta(uv)$ is a path in $G$ with ends $\eta(u)$ and $\eta(v)$ for every $uv \in E(H)$, and the paths corresponding to distinct edges are internally vertex disjoint.
\end{itemize}

For a subgraph $H' \subseteq H$ we denote by $\eta(H')$ the subgraph of $G$ with the vertex set $\eta(V(H')) \cup (\cup_{e \in  E(H')}V(\eta(e))$ and edge set $\cup_{e \in E(H')}E(\eta(e))$. (Most often we use this notation when $H'=H$.)  Note that $\eta(H)$ is a subdivision of $H$. Thus there exists an embedding of $H$ into $G$ if and only if $G$ contains a subgraph isomorphic to a subdivision of $H$, which in turn is equivalent to $H$ being a topological minor of $G$.  If $H$  maximum degree at most three, e.g. $H$ is a wall, then this is in turn equivalent to $H$ being a minor of $G$, implying the following well known observation.

\begin{obs}\label{o:embedding}
	Let $H$ and $G$ be graphs such that  the maximum degree of $H$ is at most three. Then $H \leq G$ if and only if there exists an embedding of $H$ into $G$.
\end{obs}

Let $\eta$ be an embedding of a  wall $W$ into  $G$. We refer to the images of top, bottom, left, right, corners and outer cycle of $W$ as  \emph{top, bottom, left, right, corners and outer cycle} of $\eta$. We use $T(\eta)$ to denote the top of $\eta$, and we similarly transfer the rest of notation, i.e. $c_{--}(\eta)$ denotes the corner of $\eta$, which is the common end of the left and the bottom of $\eta$. 
We use $\partial\eta$ to denote the vertex set of the outer cycle of $\eta$. 

Let $D$ be a cycle in $\eta(W)$, and let $D'$ be a cycle in $W$ such that $\eta(D')=D$. Let $X'$ be the set of vertices of $W$ which belong to the disk bounded by $D'$, and let $X$ be the set of all the vertices of $G$ which share a component of $G \setminus V(D)$ with a vertex of  $\eta(X')$. 
We refer to $G[V(D) \cup X]$ as the \emph{compass of $(\eta,D)$} an denote it by $\uparrow(\eta,D)$. Informally, in the circumstances when this notation is used, $\uparrow(\eta,D)$ corresponds to the subgraph of $G$ located inside $D$, where $\eta$ is used to distinguish inside and the outside.
The \emph{compass of $\eta$}, denoted by $\uparrow \eta$, is the compass of $(\eta,C)$, where $C$ is the outer cycle of $\eta$.

An \emph{$\eta$-cross} is a pair of vertex disjoint paths $\{P_1,P_2\}$ in  $\uparrow \eta$ such that  for  $i=1,2$ the path $P_i$ has a pair of opposite corners of $\eta$ as its ends. We say that $\eta$ is \emph{crossless} if there is no $\eta$-cross in $G$. See Figure~\ref{f:wall} for an illustration of the definitions introduced so far in this section.

\begin{figure}
	\ctikzfig{Wall2}	
	\caption{An embedding $\eta$ of a $3 \times 6$-wall with the outer cycle, indicated in red, and an $\eta$-cross, indicated in blue. }
	\label{f:wall}
\end{figure}

The Flat Wall theorem, originally proved by Robertson and Seymour~\cite{GM13}, informally states that if $G$ does not contain a fixed complete graph as a minor, then for any embedding $\eta$ of a sufficiently large wall $W$ into $G$ there exists a wall $W' \subseteq W$ of any prescribed size such that the restriction of $\eta$ to $W'$ is a crossless embedding in to $G \setminus X$ for some $X \subseteq V(G)$ of bounded size. The main external tool in our proof of \cref{t:large}  is the following technical asymmetric version of this theorem, which is implicit in a recent new proof of the Flat Wall theorem due to Chuzhoy~\cite{Chuzhoy14}.

\begin{thm}[\protect{\cite[Section 7]{Chuzhoy14}}]\label{t:flat}
	For all $t,k,h$ there exists $N=N_{\ref{t:flat}}(t,k,h)$ satisfying the following. Let $G$ be a graph such that $K_t \not \leq G$, and let $\eta$ be an embedding of the wall $W([N],[h+4t])$ into $G$. Then there exists $A \subseteq V(G)$ with $|A| \leq t-5$ and  integers $a_1,a_2,\ldots,a_k \in [N-h]$ such that $|a_i-a_j| > 2h$ for all $i \neq j$ and the restriction of $\eta$ to $W([a_i+1, a_i+2h],[2t+1,2t+h])$ is a crossless embedding of it into $G \setminus A$.
\end{thm}

Let $X,Y,Z \subseteq V(G)$ be pairwise disjoint. We say that $Y$ \emph{separates $X$ and $Z$} if $X$ and $Z$ belong to vertex sets of different components of $G \setminus Y$, i.e. $\kappa_{G \setminus Y}(X,Z)=0$. In addition to \cref{t:flat} we need the following  lemma from ~\cite{Chuzhoy14}.  

\begin{lem}[\protect{\cite[Theorem 2.9]{Chuzhoy14}}]\label{l:separate}
	Let $\eta$ be a crossless embedding of a wall $W$ into a graph $G$. Let $D$ be a cycle in $\eta(W)$ such that $V(D) \cap \partial \eta = \emptyset$. Then $V(D)$ separates $V(\uparrow(\eta,D)) - V(D)$ from $\partial \eta$.
\end{lem}

We say that a wall $W'$ is a \emph{proper subwall} of the wall $W$ if $W' \subseteq W$ and $W'$ is disjoint from the outer cycle of $W$.  The last auxiliary result about crossless embeddings of walls that we need is the following. 
  
\begin{lem}\label{l:crossless}
 	Let $\eta$ be a crossless embedding of a wall $W$ into a graph $G$. Then the restriction of $\eta$ to any  proper subwall of $W$ is crossless.
 \end{lem}
 
\begin{proof}
	Let $W'$ be a proper subwall of $W$. It is easy to check that there exists a linkage $\mc{P}$ of size $4$ in $W$ such that the paths in $\mc{P}$ are internally disjoint from $V(W')$ and each path of $P$ joins a corner of $W'$ to the corresponding
	corner of $W$.
	
	Let $\eta'$ denote the restriction of $\eta$ to $W'$. Suppose for a contradiction that there exists an $\eta'$-cross $\{P_1,P_2\}$ in $G$.  Let $\mc{Q}= \{\eta(P)\}_{P \in \mc{P}}$ be the image of the linkage $\mc{P}$. By \cref{l:crossless}, $\partial \eta'$ separates $(V(P_1) \cup V(P_2)) - \partial \eta'$ from $\partial \eta$. It follows that the paths in $\mc{Q}$ are disjoint from $P_1$ and $P_2$ except for the corners of $\eta'$, which they share as ends. Thus we can use $\mc{Q}$ to extend $\{P_1,P_2\}$  to an $\eta$-cross, a contradiction. 
\end{proof}

The proof of \cref{t:large} occupying the rest of the section can be informally outlined as follows.
We find the patchwork $\mc{H}$ satisfying \cref{t:large} by finding a single well-behaved patch embedded deep within each of the crosslessly embedded walls guaranteed by \cref{t:flat} and use the rest of the embedding to stitch the patches together. 
By choosing initial patches far from the boundary of the walls, we are able to link the boundary of the patch to the corners of the corresponding wall, utilizing the cycles in the  wall that separate these two boundaries to route the linkage we are producing according to our needs.

To formalize this last part of the argument, 
we say that a collection of pairwise vertex disjoint cycles $\mc{C}$ in $G$ is an \emph{$(X,Z,|\mc{C}|)$-filter (in $G$)} if $V(C)$ separates $X$ and $Z$ for each $C \in \mc{C}$.

\begin{lem}\label{l:filter}
	Let  $G$ be a graph, let $X',X,Z \subseteq V(G)$ be such that $X' \subseteq X$.  
	Let $(A,B)$ be a separation $G$ such that $Z \cap A \neq \emptyset,$ $X \subseteq B$, 
	and $G[A]$ is connected.  Let $r \geq 1$ be an integer such that $\kappa_G(X',Z) \geq r$, and there exists an $(X,Z,r)$-filter in $G$. \begin{itemize}
		\item[(i)] If $|A \cap B| < r$ then there exists a cycle $C$ in $G$ such that $V(C)$ separates $A$ and $X$, in particular $X \cap A = \emptyset$,
		\item[(ii)] If $\kappa_G(A \cap B,X)  \leq r$ then $\kappa_G(A \cap B,X') = \kappa_G(A \cap B,X)$.
		\end{itemize}
\end{lem}

\begin{proof} Let $\mc{C}$ be an  $(X,Z,r)$-filter in $G$.  
	
	We prove (i) first. As $|A \cap B| < r$ there exists $C \in \mc{C}$ such that $V(C) \cap A \cap B = \emptyset$. 
	As  $\kappa(X,Z) \geq r$ it follows that $\kappa(X,V(C)) \geq r$ and so $V(C) \not \subseteq A$. Thus $V(C) \cap A = \emptyset$.
	As $G[A]$ is connected, $A$ belongs to the vertex set of a single component of $G \setminus V(C)$ and this  component intersects $Z$, and so is disjoint from $X$.  Thus $V(C)$ separates $A$ and $X$, as desired
	
	It remains to prove (ii). Let $r' =\kappa_G(A \cap B,X)$. Suppose for a contradiction that there exists a separation $(A',B')$ such that $A \cap B \subseteq A'$, $X' \subseteq B'$ and $|A' \cap B'| < r' \leq r.$ 
	As  $|\mc{C}| >  |A' \cap B'|$, there exists $C \in \mc{C}$ such that $V(C) \subseteq A'-B'$, or $V(C) \subseteq B'-A'$. As $V(C)$ separates $X'$ and $Z$ we have $\kappa_G(X',V(C)) \geq \kappa_G(X',Z) \geq r > |A' \cap B'|$. As $X' \subseteq B'$, it follows that $V(C) \not \subseteq A'$,  and so    $V(C) \subseteq B'-A'$. However, $A \cap B \subseteq A'$, and $\kappa_G(A \cap B,V(C)) \geq  \kappa_G(A \cap B,X) = r' > |A' \cap B'|$, a contradiction.
\end{proof}

We now define the patches that we will be searching for within walls. We say that a  $q$-patch $H$ is a \emph{diamond patch} if  \begin{description}
	\item(D1) $q \in [3]$,
	\item (D2) $\nm{H}$ is connected,
	\item(D3) $a_H(i) \neq b_H(i)$ if and only if $i = 1$,
    \item(D4) There exist pairwise vertex disjoint connected subgraphs $(S_1,T_1, S_2, \ldots, S_q)$ of $\nm{H}$ such that
    \begin{itemize}
    	\item $a_H(i) \in V(S_i)$ for every $i \in [q]$ and $b_H(1) \in T_1$,
    	\item There exists an edge of $H$ with one end in $V(S_1)$ and the other in $V(T_1)$
    	\item There exists an edge of $H$ with one end in $V(S_1)$ and the other in $V(S_i)$, and an edge of  $H$ with one end in $V(T_1)$ and the other in $V(S_i)$ for each $2 \leq i \leq q$. 
    \end{itemize}
\end{description}

Let $\eta$ be an embedding of  an $(l,h)$-wall $W$ into a graph $G$. The \emph{pegs} of $\eta$ are the images of vertices in the interior of $L(W)$ and $R(W)$ of degree two. Note that if $\eta$ is a restriction of an embedding $\eta$ of a wall $W'$ into $G$ such that $W$ is a proper subwall of $W'$ then the pegs of $\eta$ are incident to edges in $E(\eta(W')) - E(\eta(W))$, which is useful for extending the paths ending in these pegs outside of $\uparrow \eta$. 

\vskip 10pt
We say that a diamond patch $H$ is \emph{properly set in $\eta$}, if 
  	\begin{description}
  		\item (DS1) $H$ is embedded in $G$,
  		\item  (DS2) $V(H) \subseteq V(\uparrow \eta) - \partial \eta,$ 
  		\item (DS3) there exists an $(\partial H,\partial \eta)$-linkage $(L_1,L'_1,L_2,\ldots,L_q)$, internally disjoint from $B(\eta)$ and $T(\eta)$,  such that  \begin{itemize} \item
  		 $L_1$ has ends $a_H(1)$ and a peg on $L(\eta)$,
  			\item
  			 $L'_1$ has ends $b_H(1)$ and a peg on $R(\eta)$, 
  			\item if $q \geq 2$ then
  		$L_2$ has  $a_H(2)$ as one end and the second end on  $B(\eta)$, and
  			\item if $q = 3$ then
  			$L_3$ has  $a_H(3)$ as one end and the second end on $T(\eta)$.
  		\end{itemize}	 
  	 \end{description}	

The next lemma is the main technical step in our construction of the patchwork satisfying \cref{t:large}.
 
\begin{lem}\label{l:diamond} Let $l,h \geq 16$ be positive integers, and let $\eta$ be
	a crossless embedding of an $(l,h)$-wall $W$ into a graph $G$. Then there exists a diamond patch $H$ properly set in $\eta$.
\end{lem}

\begin{proof} Let $x_0,y_0$ be integers such that $W = W(X,Y)$ for  $X=[x_0+1,x_0+2l], Y=[y_0+1,y_0+h]$. For every $i \in  [0,7]$, let $X_i=[x_0+2i+1,x_0+2l-2i],Y_i=[y_0+i+1,y_0+h-i],$ $W_i=W(X_i,Y_i)$, and let $C_i$ be the outer cycle of the restriction $\eta_i$ of $\eta$ to the wall $W_i$. Then $W_j$ is a proper subwall of $W_i$ and $C_j \subseteq \uparrow \eta_i$ for all $0 \leq i \leq  j \leq 7$. Thus by ~\cref{l:crossless}, $V(C_j)$ separates $V(C_i)$ and $V(C_k)$ for all $0 \leq i < j < k \leq 7$. It follows that $\{C_i\}_{i \in [6]}$ is a $(\partial \eta,V(C_7),6)$-filter. It is easy to see that $\kappa(V(C_0),V(C_7)) \geq 6$.
	
Choose a diamond patch $H$ such that \begin{description}
\item (H1) $\nm{H}=G[V(H)]$ and $(V(H), V(G)-(V(H) -\dH))$ is a separation of $G$,
\item (H2) $V(H) \cap V(C_7) \ne \emptyset$,
\item (H3)  $\kappa(\partial H,\partial \eta)=|\partial H|$,
\end{description} and $|V(H)|$ is maximum, subject to the conditions (H1)-(H3).   
Note that such a choice is possible as a  $1$-patch $H$ with $V(H)$ consisting of the ends of any edge of $C_7$ satisfies the requirements. 

In the rest of the proof we show that some patch obtained from $H$ by relabeling the boundary is properly set in $\eta$. First, we show that every diamond patch $H$ chosen as above satisfies (DS1) and (DS2).

Let $q \in [3]$ be such that $H$ is a $q$-patch. \cref{l:filter} (i) applied to $r= q+2$, $X =  \partial \eta, Z = V(C_7), A = V(H)$,  and  $B = V(G)-(V(H) -\dH)$ implies that some cycle $C'$ in $G$ separates $V(H)$ and $\partial \eta$. In particular, (DS2) holds. 
  
Suppose next, for a contradiction that (DS1) does not hold. By (H1), it follows that  there exists $v \in V(G) - V(H)$ such that $v$ is adjacent to $a_H(1)$ and $b_H(1)$.  The existence of a cycle $C'$ as above  implies that $v \in V(\uparrow \eta)-\partial \eta$.
	
	Next we show that $\kappa(S,\partial \eta)=|S|=q+2$, where $S =\dH \cup \{v\}$. Suppose not. Then, as $\kappa(\partial H,\partial \eta)=|\dH|=q+1$ by (H3), there exists a separation $(A,B)$ of $G$ of order $q+1$ such that $G[A]$ is connected, $V(H) \cup \{v\} \subseteq A$, $\partial \eta \subseteq B$.  It further follows that there exists a $(\dH,A \cap B)$-linkage $\mc{P}$ in $G[A]$ with $|\mc{P}|=q+1$. Let $H'$ be a $q$-patch such that $\nm{H'}=G[A]$, $a_{H'}(i)$ is the second end of the path in $\mc{P}$ with an end $a_H(i)$,  and   $b_{H'}(i)$ is the second end of the path in $\mc{P}$ with an end $b_H(i)$, for every $i \in [q]$.  
	
	Let $(S_1,T_1, S_2, \ldots, S_q)$ be the subgraphs of $H$ satisfying (D4). For each of these subgraphs there exists a unique  path in $\mc{P}$ which has an end in it. By taking the union of each subgraph with the respective path we obtain subgraphs  $(S'_1,T'_1, S'_2, \ldots, S'_q)$ of $H'$, which satisfy (D4) for $H'$. It follows that $H'$ is a diamond $q$-patch. 
	Note that the conditions (H1)--(H3) hold for $H'$ and $V(H) \cup \{v\} \subseteq V(H')$. Thus  $H'$ violates the choice of $H$, yielding the desired contradiction. 
	 
	Thus $\kappa(S,\partial \eta)=|S|$. Consider now a $(q+1)$-patch $H''$ with  $\nm{H}'' = G[V(H) \cup \{v\}]$, $a_{H''}(i)=a_{H}(i)$ and $b_{H''}(i)=b_{H}(i)$ for all $i \in [q]$, and $a_{H''}(q+1)=b_{H''}(q+1)=v$. Note that the conditions (H1)--(H3) hold for $H''$, although we have not established yet that $H''$ is a diamond patch.
	
	Let $S_{q+1}$ be the one vertex graph with $V(S_{q+1})=\{v\}$. Then $(S_1,T_1, S_2, \ldots, S_{q+1})$ satisfies the condition (D3) for $H''$. Thus, if $q < 3$, then $H''$ is a diamond patch, contradicting the choice of $H$.
	
	It remains to consider the case $q=3$.  Let $X'=\{x_1,x_2,x_3,x_4,x_5\} \subseteq \partial \eta$ be a set of non-corner  vertices such that $$(c_{--}(\eta), x_1, c_{-+}(\eta),x_2,x_3,c_{++}(\eta),x_4,c_{+-}(\eta), x_5)$$
	appear on the outer cycle of $\eta$ in the order listed, and each $x_i$ is incident to an edge in $\eta(W)$ with the second end outside $\partial \eta$. 
	\begin{figure}
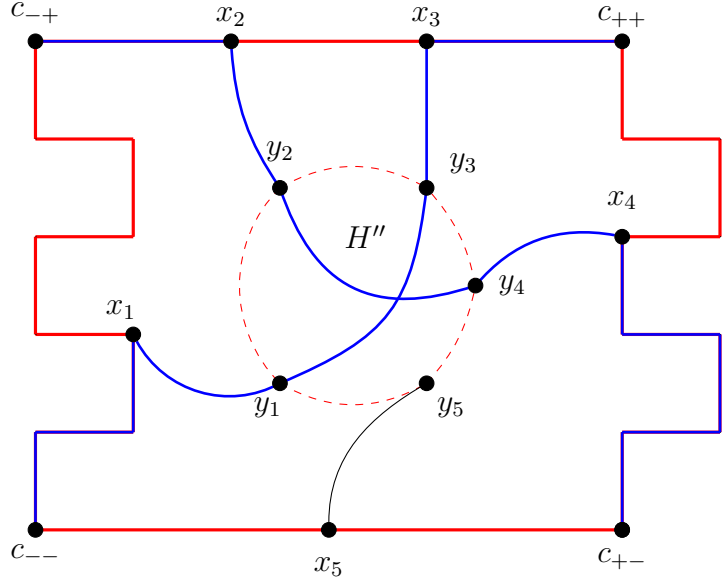

		\ctikzfig{Diamond2}
		\caption{An $\eta$-cross obtained by extending two paths in $\nm{H''}$ (one with ends $y_1$ and $y_3$, and the other with ends   $y_{2}$ and $y_{4}$)  using $\mc{P}'$ and the outer cycle of $\eta$.} \label{f:diamond2}
	\end{figure}
	
	Let $G' = (\uparrow \eta) \setminus (\partial \eta - X')$. Let $r= 5, X = V(C_1) \cup X'$, $Z = V(C_7), A = V(H'')$  and  $B = V(G')-(V(H'') -\dH'')$. It is easy to see that $\kappa_{G'}(X',Z) = 5$ and $\{C_2,\ldots,C_6\}$ is an $(X,Z,5)$-filter in $G'$. Thus   $\kappa_{G'}(X', \partial H'') = 5$ by \cref{l:filter} (ii). Let $\mc{P}'$ be an $(X',\partial H'')$- linkage in $G'$ with $|\mc{P}'|=5$. Let $y_i$ denote the end of a path in $\mc{P}'$ with the second end $x_i$ for every $i \in [5]$. Suppose that for some $j \in [5]$ there exists a pair of vertex disjoint  paths in $\nm{H''}$, one with ends $y_{j}$ and $y_{j+2}$, and the other with ends   $y_{j+1}$ and $y_{j+3}$, where the indices are taken modulo five.
	Using $\mc{P}'$ these paths can be extended to a pair of vertex disjoint paths, internally disjoint from $\partial \eta$, one with ends $x_{j}$ and $x_{j+2}$, and the other with ends    $x_{j+1}$ and $x_{j+3}$. These can be further extended along the outer cycle of $\eta$ to an $\eta$-cross. See Figure~\ref{f:diamond2}  for an example.
	
	As $\eta$ is crossless, no pair of paths as above exists. Let $Z'= \{a_{H''}(2),a_{H''}(3), a_{H''}(4)\}$. It follows from the condition (D3) that  for every $j \in [5]$ either $\{y_{j}, y_{j+2}\} \subseteq Z'$, or $\{y_{j+1}, y_{j+3}\} \subseteq Z'$. This is clearly impossible, yielding the desired contradiction. This finishes the verification of the condition (DS1) for $H$.
	
It remains to find the $(\partial H,\partial \eta)$-linkage satisfying (DS3). Let $X''=\{x_1,\ldots,x_{4}\}\subseteq \partial \eta$ be a set of  vertices such that $$(c_{--}(\eta), x_1, x_4, c_{-+}(\eta),c_{++}(\eta),x_3, x_2,c_{+-}(\eta))$$ appear on the outer cycle of $\eta$ in the order listed, in particular  $x_1, x_4 \in V(L(\eta))$, $x_2,x_3 \in V(R(\eta))$, and, as before, each $x_i$ is incident to an edge in $\eta(W)$ with the second end outside $\partial \eta$.  Let $X' =\{x_1,\ldots,x_{q+1}\}$. By \cref{l:filter} applied to the graph $G'=\uparrow \eta \setminus (\partial \eta - X')$ with $r=q+1$,  $X = V(C_1) \cup X'$,  $Z = V(C_7), A = V(H)$  and  $B = V(G')-(V(H) -\dH)$, we have  $\kappa_{G'}(X', \partial H) = q+1$. Let the linkage $\mc{P}'$ in this case  be an $(X',\partial H)$- linkage in $G'$ with $|\mc{P}'|=q+1$. As before, we define $y_i$ to be the end of a path in $\mc{P}'$ with the second end $x_i$ for every $i \in [q+1]$.

Suppose first that $q \leq 2$. Consider the patch $H^{*}$ such that $\nm{H^{*}}=\nm{H}$, $a_{H^{*}}(1)=y_1,b_{H^{*}}(1)=y_2$, and, if $q=2$, $a_{H^{*}}(2)=b_{H^{*}}(2)=y_3$. Then $H^{*}$ is also a diamond patch, satisfying (H1)-(H3). Thus, as we have already shown,  $H^{*}$ satisfies
(DS1) and (DS2). Moreover, the paths in $\mc{P}'$ starting in $y_1,y_2$ and $y_3$ can be extended along the outer cycle of $W$  to a peg on $L(\eta),$ a peg on $R(\eta)$, and $c_{--}(\eta)$, respectively, producing a linkage satisfying (DS3).

Finally, we finish the proof in the case $q=3$. Let $i \in [4]$ be such that $a_H(1)=y_i$. Suppose that $b_H(1) \neq y_{i+2}$, where the indices of $y$ vertices here and below are taken modulo four. Then by (D3) there exist vertex disjoint paths $P,Q \subseteq H$ such that $P$ has ends $y_i$ and $y_{i+2}$  and $Q$ has ends $y_{i+1}$ and $y_{i+3}$. Extending these paths using $\mc{P}'$ to the outer cycle of $\eta$, and along the outer cycle to the corners, we obtain an $\eta$-cross, a contradiction.

Thus $b_H(1) = y_{i+2}$. As in the previous case, we may switch $a_H(1)$ and $b_H(1)$, and so we  assume that $a_H(1) \in \{y_1,y_4\}$.
If $a_H(1)=y_1$ then $b_H(1)=y_3$, and by switching $a_H(2)$ and $a_H(3)$, if needed, we may assume $a_H(2)=y_2$ and $a_H(3)=y_4$. Extending the paths in  $\mc{P}'$ starting in $y_1,y_2,y_3$ and $y_4$ along the outer cycle of $\eta$ to a peg on $L(\eta),$ $c_{+-}(\eta)$, a peg on $R(\eta)$, and $c_{-+}(\eta)$, we obtain the linkage satisfying (DS3). See \cref{f:diamond}. If $a_H(1)=y_4$  then $b_H(1)=y_2$, we switch the labels so that $a_H(2)=y_1$ and $a_H(3)=y_3$, and extend $\mc{P}'$ along the outer cycle of $\eta$ from $y_1,y_2,y_3$ and $y_4$ to $c_{--}(\eta)$,  a peg on $R(\eta)$, $c_{++}(\eta)$, and a peg on $L(\eta),$ respectively,  obtaining the linkage satisfying (DS3) in this last case.
\end{proof}	
 
\begin{figure}
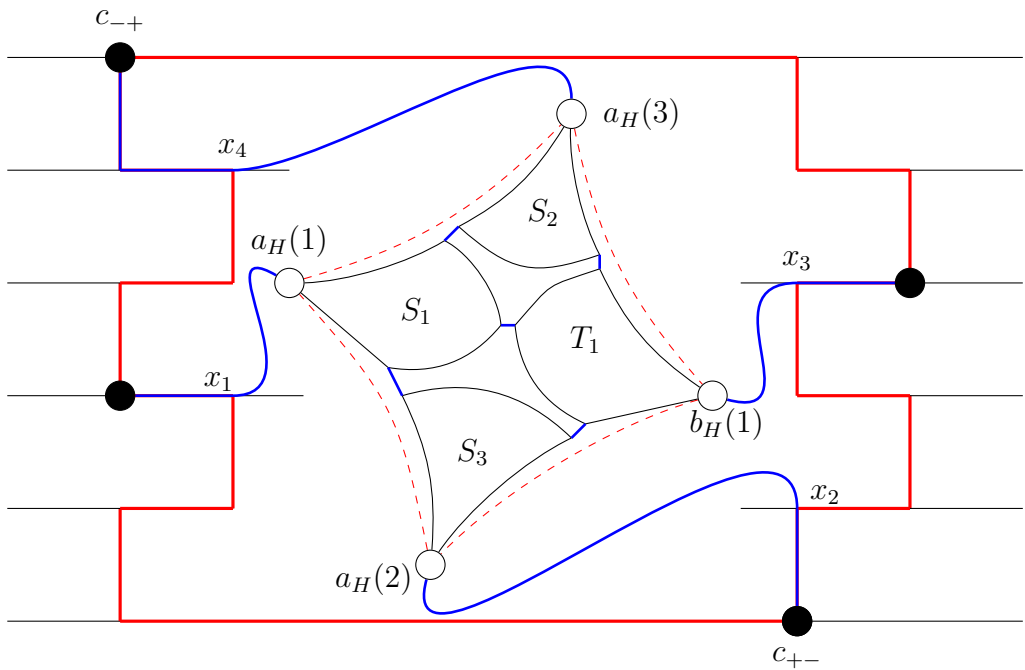

\ctikzfig{Diamond}
\caption{A diamond $3$-patch $H$ properly set in an embedding of a wall with a linkage satisfying (DS3) indicated in blue.} \label{f:diamond}
\end{figure}

We are now ready to derive \cref{t:large} from \cref{t:flat} and \cref{l:diamond}.

\begin{proof}[Proof of  \cref{t:large}]
	Let $h=18$, $s = 4t+h$, $k=6(t-5)l$, and let $n=N_{\ref{t:flat}}(t,k,h)$.  Let $G$ be a graph with an $(n,s)$-grid minor such that $K_t \not \leq G$. As $W([n],[s])$ is a subgraph of the $(n,s)$-grid we have $W([n],[s]) \leq G$, and so by \cref{o:embedding} there exists an embedding $\eta$ of $W([n],[s])$ into $G$.  By \cref{t:flat} there exists  $A \subseteq V(G)$ with $|A| \leq t-5$ and  integers $a_1,a_2,\ldots,a_k \in [N-h]$ such that $|a_i-a_j| > 2h$ for all $i \neq j$ and the restriction of $\eta$ to $W([a_i, a_i+2h-1],[2t,2t+h])$ is a crossless embedding into $G \setminus A$. Let $G' = G \setminus A$
	
	For every $i \in [k]$, let $W'_i = W([a_i, a_i+2h-1],[2t,2t+h])$, and let $W_i=W([a_i+2, a_i+2h-3],[2t+1,2t+h-1])$. Let $\eta'_i$ and $\eta_i$ denote the restrictions of $\eta$ to $W'_i$ and $W_i$, respectively,  considered as embeddings into $G'$.
	
	 By \cref{l:crossless} the embedding $\eta_i$  is   crossless. By \cref{l:diamond} there exists $q_i \in [3]$ and a diamond $q_i$-patch $H_i$ properly set in $\eta_i$. By \cref{l:separate},
	$\partial \eta_i$ separates $V(\uparrow \eta_i) - \partial \eta_i$ from $\partial \eta'_i$ in $G'$. 
	
	By the pigeonhole principle, reordering $a_1,a_2,\ldots,a_k$ if needed we may assume that
	\begin{itemize}
		\item $a_i < a_{i+1} - 4h$ for $i \in [l-1]$
		\item $q_1=q_2 = \ldots = q_l =q'$ for some $q \in [3]$,
		\item $V(\eta(W([a_1,a_l+2h-1],[s]))) \cap A = \emptyset.$
	\end{itemize} 
	
	We claim that $\mc{H}=(H_1,\ldots,H_l)$ is a refined $q'$-patchwork stitched to $G'$. Note that this claim implies the theorem, as we can add the vertices of $A$ to all patches of $\mc{H}'$ to obtain a refined $(q'+|A|)$-patchwork stitched to $G$. More formally, let $q = q'+|A|$ and let $A= \{v_{q'+1},v_{q'+2}, \ldots, v_q\}$. We transform every patch $H_i$ into a $q$-patch $H'_i$ by defining $\nm{H'_i}=G[V(H_i) \cup A]$, and extending the maps $a_{H_i}$ and $b_{H_i}$ to $a_{H'_i}$ and $b_{H'_i}$ by setting $a_{H'_i}(j)=b_{H'_i}(j)= v_j$ for all $j \in [q'+1,q]$. It can now be routinely verified that 
	$(H_1,\ldots,H_l)$ is a refined $q$-patchwork stitched to $G$.
	
	It remains to prove the claim. A diamond patch is refined by definition, and so $\mc{H}$ is refined. Next we verify that 
	$\mc{H}$ is embedded in $G'$. The property (E1) holds by the property (DS1) of properly set patches. 
	For all $\{i,j\} \subseteq  [l]$,  $V(\uparrow \eta_j)$ belongs to the same  component of  $G' \setminus \partial \eta_i$ as  $\partial \eta'_i$, and therefore $\partial \eta_i$ separates $V(\uparrow \eta_i)- \partial \eta_i$ and $V(\uparrow \eta_j)$.  Thus by (DS2), $\partial \eta_i$ of $W_i$ separates $V(H_i)$ and $V(H_j)$,  implying (E2), (E3) and (E4).
	
\begin{figure}
	\ctikzfig{Stitch}
	\caption{A fragment of the linkage $\mc{P}$ stitching the patchwork $\mc{H}$ to $G'$.} \label{f:stitch}
\end{figure}		

	It remains to construct a collection of paths $\mc{P}=(P_{i,j})_{i \in [q'], j \in [l-1]}$ stitching $\mc{H}$ to $G'$. 
	We assume $q'=3$ for clarity, while our argument can be trivially adapted to the cases $q'=1,2$, by omitting some of the paths. Let $(L_{1,j},L'_{1,j},L_{2,j},L_{3,j})$ be the linkage satisfying (DS3) for $H_j$ and $\eta_j$. Let $R_k$ denote the image of the path in $W$ consisting of all the vertices with $y$-coordinate equal to $k$.
	 We define \begin{itemize} \item $P_{1,j}$ to be the union of paths $L'_{1,j}, L_{1,j+1}$ and a path in $$\eta(W([a_j+2h-3, a_{j+1}], [2t+2,2t+h-2]))$$ joining the ends of these two paths;
	\item $P_{2,j}$ to be the union of paths $L_{2,j}, L_{2,j+1}$ and a path in $R_{2t+1}(\eta)$ joining their ends;
	\item $P_{3,j}$ to be the union of paths $L_{3,j}, L_{3,j+1}$ and a path in $R_{2t+h-1}(\eta)$ joining their ends.
	\end{itemize} 
	See Figure \ref{f:stitch}.
	It is easy to see that $\mc{P}$ is as required.
\end{proof}

\section{Topological minors}\label{s:top}

In this section we prove \cref{t:topological}. For graphs $H$ and $G$, we write $H \leq_t G$ if $H$ is a topological minor of $G$.

We start with \cref{t:topological} (i). Eppstein's proof~\cite{Eppstein10} of the fact that densities of minor-closed classes are well-ordered uses the Robertson-Seymour graph minor theorem.  We substitute it with the following theorem by Liu and Thomas~\cite{LiuTho20}. 

To state this theorem we first need to introduce the family of graphs which play the key role in well-quasi-ordering of topologically minor-closed families, defined as follows.
A \emph{simple Robertson chain  of length $k$}, denoted by $RC(k)$, is the graph obtained from a path $P$ with $k$ edges by adding for every $e \in E(P)$ a vertex $v_e$ of degree two adjacent to ends of $e$. That is $RC(k)$ consists of $k$ triangles glued in a linear fashion along one vertex cutsets. 

Let $RC'(k)$ be obtained from $RC(k)$ by additionally attaching three leaves to each of the ends of $P$. It is not hard to see that the graphs $\{RC'(k)\}_{k \in \bb{N}}$ are pairwise incomparable in the topological minor order. Further  families  with this property can be obtained from $\{RC(k)\}_{k \in \bb{N}}$ in a similar manner.

Conversely, Robertson conjectured that all obstructions to well-quasi-ordering by the topological minor  relation involve Robertson chains. This conjecture, and its strengthening to labeled minors along the lines of \cref{t:RSwqo}, were proved by Liu in his PhD thesis~\cite{Liu14}. We only need a restricted version of this result, which is a weakening of ~\cite[Theorem 1.3]{LiuTho20}   

\begin{thm}[\cite{LiuTho20}]\label{t:topwqo}
	Let $k$ be a positive integer and let $L$ be a finite set. For every $i \in \bb{N}$, let $(G_i)_{i \in \bb{N}}$ be a graph such that $\tw(G_i) \leq k$ and $RC(k) \not \leq_t G_i$ and let $\phi_i: V(G_i) \to L$ be a function.   Then for some $i < j$ there exists an embedding $\eta$ of $G_i$ into $G_j$ such that $\phi_j(\eta(v))=\phi_i(v)$ for every $v \in V(G_i)$.
\end{thm}

Extending the definition of topological minors to patches, we say that a $q$-patch $H$ is a \emph{topological minor} of a $q$-patch $G$, and write $H \leq_t G$ if there exists an embedding $\eta$ of $\nm{H}$ into $\nm{G}$ such that $\eta(a_H(i))=a_G(i)$ and   $\eta(b_H(i))=b_G(i)$ for every $i \in [q]$. It is easy to verify that, similarly the minor order, the topological minor order is well-behaved with respect to patch products:  if $H_1,G_1,H_2$ and $G_2$ are patches $H_1 \leq_t G_1$ and $H_2 \leq_t G_2$, then $H_1 \times H_2 \leq_t G_1 \times G_2$. In the same way that \cref{t:wqo} is derived from \cref{t:RSwqo}, we obtain from \cref{t:topwqo} the following.

 \begin{cor} \label{c:topwqo}
 	Let $k \geq 1$, $q \geq 0$ be integers.  For every $i \in \bb{N}$, let $H_i$ be a $q$-patch such that $\tw(\nm{H_i}) \leq k$ and $RC(k) \not \leq_t \nm{H_i}$.  Then $H_i \leq H_j$ for some $i < j$.
 \end{cor}

We say that a $q$-patch $H$ is a \emph{fan} patch, if $H$ is non-degenerate, $a_H(i)=b_H(i)$ for every $i \in [q]$, $\nm{H} \setminus \dH$ is connected, and every vertex  of $\dH$ has a neighbor in $V(H)-\dH$. 
Eppstein~\cite[Lemma 17]{Eppstein10} proved that the limiting density of a proper minor-closed family $\mc{F}$ can be approximated with arbitrary precision by powers of $\mc{F}$-tame fan patches.\footnote{In comparison, in \cref{t:pathwidth} we attain $d(\mc{F})$ exactly, but we can not restrict our attention to fan patches.} Importantly, the proof of~\cite[Lemma 17]{Eppstein10}  extends to 
 any monotone\footnote{A graph family is \emph{monotone} if it is closed under subgraphs and isomorphisms.} subfamily of a proper minor-closed graph family, giving the following.

\begin{lem}[\protect{\cite[Lemma 17]{Eppstein10}}]\label{l:fan}
	Let $\mc{F}$ be a monotone graph family such that $\mc{F} \subseteq \mc{F}'$ for some proper minor-closed family $\mc{F}'$. Then for any $\eps > 0$ there exists an $\mc{F}$-tame fan patch $H$ such that $$\lim_{n \to \infty}d(\nm{H^n}) \geq d(\mc{F}) -\eps.$$
\end{lem} 

The next simplelemma shows that for families of density less than two, we can further assume that the patch $H$ satisfying the conditions of \cref{l:fan} is a $0$-, or $1$-patch and that \cref{c:topwqo} is applicable to $H$.

\begin{lem}\label{o:fan2} For every $\delta >0$ there exists $k = k_{\ref{o:fan2}}(\delta)$ satisfying the following. Let $\mc{F}$ be a  family of graphs closed under taking topological minors with $d(\mc{F}) \leq 3/2 - \delta$, and let  $H$ be an $\mc{F}$-tame fan $q$-patch. Then $q \leq 1$, $\tw(\nm{H})  \leq k$ and $RC(k) \not \leq_t \nm{H}$.
\end{lem}	 

\begin{proof} Suppose first, for a contradiction, that $H$ is  an $\mc{F}$-tame fan $q$-patch for some $q \geq 2$. By definition of a fan patch, there exists a path $P$ in $H$ with ends $a_H(1)$ and $a_H(2)$, $|V(P)| \geq 3$ and $P$ is internally disjoint from $\dH$. By deleting vertices in $V(H) - \partial H - V(P)$ and contracting all, but two edges of $P$ we obtain a topological minor  $H'$ of $H$ such that $|V(H')|=q+1$ and $e(H')=2$. Then $H'$ is $\mc{F}$-tame and  By \cref{c:limitdensity} we have $$d(\mc{F}) \geq \lim_{n \to \infty}d(\nm{(H')^n}) \geq 2,$$
	a contradiction.
	
Thus $q \leq 1$. Let $G = \nm{H} - \partial H$.  Note that $G$, considered as a $0$-patch is $\mc{F}$-tame, and so \begin{equation}\label{e:gprime} d(G') \leq d(\mc{F}) \leq \frac{3}{2}-\delta \end{equation} for every topological minor $G'$ of $G$. 

Let $GR(w)$ denote the $(w,2)$-grid. It is well-known that for every positive integer $w$ there exists $k=k(w)$ such that $GR(w) \leq J$ for every graph $J$ with $\tw(J) \leq k$. By \cref{o:embedding} this in turn implies, $GR(w) \leq_t J$ as the maximum degree of $GE(w)$ is at most $3$. Let an integer $w \geq 1$ be chosen so that $\frac{3w-5}{2w-1} > 3/2 - \delta$, and let $k \geq k(w)$ be chosen so that $\frac{3k-4}{2k} > 3/2 - \delta$. 

It is easy to see from the above observation that such $k$ satisfies the lemma. Indeed, if $\tw(\nm{H}) \geq k$ then $GR(w) \leq_t \nm{H}$ and so $G$ contains as a topological minor a graph $G'$ obtained from $GR(w)$ by deleting one vertex. We have
$$d(G') \geq \frac{|E(GR(w))| - 3}{|V(GR(w))|-1} = \frac{3w-5}{2w-1} > \frac{3}{2}-\delta,$$ 
contradicting \eqref{e:gprime}. Thus $\tw(\nm{H}) \leq k$.

If $RC(k) \leq_t \nm{H}$ we obtain the contradiction in the same way, and so $RC(k) \not \leq_t \nm{H}$
\end{proof}	

\cref{c:topwqo} and Lemmas ~\ref{l:fan} and~\ref{o:fan2} imply  \cref{t:topological} (i)  as follows.

\begin{proof}[Proof of \cref{t:topological} (i).]
	Suppose for a contradiction that $\mc{D}_T \cap [0,3/2]$ is not well-ordered. Then there exist families $\mc{F}_1,\mc{F}_2,\ldots, \mc{F}_n \ldots$ closed under taking topological minors such that $3/2 > d(\mc{F}_1) > d(\mc{F}_2) > \ldots> d(\mc{F}_n)> \ldots$.
	Let $\delta= 3/2- d(\mc{F}_1)$, and let $k = k_{\ref{o:fan2}}(\delta)$.
	
By \cref{l:fan} for every $i \in \bb{N}$ there exist $\mc{F}_i$-tame fan $q_i$-patch $H_i$ such that $\lim_{n \to \infty}d(\nm{H_i^n}) > d(\mc{F}_{i+1}).$
By \cref{o:fan2} we have $q_i \leq 1$,  $\tw(\nm{H})  \leq k$ and $RC(k) \not \leq_t \nm{H}$. By restricting to a subsequence of $(\mc{F}_i)_{i \in \bb{N}}$ we assume without loss of generality that there exists $q \in \{0,1\}$ such that $q_i=q$ for every $i$.

By \cref{c:topwqo} there exist integers $1 \leq i < j$ such that  $H_i \leq_t H_j$. Thus $H_i$ is $\mc{F}_j$-tame, implying $$d(\mc{F}_{i+1}) < \lim_{n \to \infty}d(\nm{H_i^n}) \leq d(\mc{F}_j) \leq d(\mc{F}_{i+1}) ,$$ a contradiction.
\end{proof}

It remains to prove   \cref{t:topological} (ii). That is for every $\Delta \geq 3/2$ we need to construct a graph family closed under taking topological minors with limiting density $\Delta$. We define such a family $\mc{F}_{\Delta}$ as follows. Define the \emph{$\Delta$-excess} of a graph $G$ as $$\psi_{\Delta}(G)=\max_{H \leq_t G, |V(H)| \geq 2} (|E(H)| - \Delta(|V(H)|-1)).$$ 
We say that $1$-patchwork $\mc{H}=(H_i)_{i \in [l]}$ is \emph{$\Delta$-controlled} if \begin{itemize}
	\item[($\Delta$1)] $\sum_{i \in J} \psi_{\Delta}(\nm{H_i}) \leq \Delta - 1$ for every interval $J \subseteq [l]$, and
	\item[($\Delta$2)]  $a_{H_i}(1) \neq b_{H_i}(1)$ for every $i \in [l]$.  
\end{itemize} 
Finally,
let $\mc{F}_{\Delta}$ denote the family consisting of all graphs $G$ such that $G$ is a topological minor of $\nm{\mc{H}}$ for some $\Delta$-controlled $1$-patchwork $\mc{H}$.

\begin{lem}\label{l:upper}
We have $|E(G)| \leq \Delta|V(G)|-1$ for every non-null graph $G \in \mc{F}_{\Delta}.$
\end{lem}
\begin{proof} Let $G$ be a topological minor of $\nm{\mc{H}}$ for a $\Delta$-controlled $1$-patchwork $\mc{H} = (H_i)_{i \in [l]}$. 
	
	We prove the lemma by induction on $l$. The base case $l=1$ is immediate from ($\Delta$1).
	
	For the induction step, by additional induction on $|V(G)|$, we may assume that $G$ is connected.  
	Let $\eta$ be an embedding of $G$ into $\nm{\mc{H}}$. If $|V(\eta(G)) \cap V(H_1)| \leq 1$, or $|V(\eta(G)) \cap V(H_l)| \leq 1$, the result follows by the induction hypothesis applied to $\mc{H}|_{[2,l]}$ or $\mc{H}|_{[l-1]}$, respectively. Thus we assume that $|V(\eta(G)) \cap V(H_i)| \geq 2$ for $i \in \{1,l\}$.
	
	For $i \in [l]$, let $G_i$ denote the maximum subgraph of  $G$ such that $\eta(G_i) \subseteq \nm{\mc{H}_i}$. Note that, if $V(G_i) \cap V(G_j) \neq \emptyset$ for some $\{i,j\} \subseteq [l]$, then $|V(G_i) \cap V(G_j)|=1$ and $|i-j|=1$ by ($\Delta$2). In particular,  \begin{equation}\label{e:vertexsum}
\sum_{i \in [l]}(|V(G_i)|-1) \leq |V(G)|-1.
	\end{equation}
	Suppose first that $G = \cup_{i \in [l]}G_i$. Then $|V(G_i)| \geq 2$ for every $i \in [l]$, by the assumptions above, and so 
	\begin{align*}|E(G)| &= \sum_{i \in [l]}|E(G_i)| \leq  \sum_{i \in [l]}(\Delta(|V(G_i)|-1) + \psi_{\Delta}(\nm{H_i})) \\ &  \stackrel{(\Delta 1)}{\leq} (\Delta-1) + \Delta\sum_{i \in [l]}(|V(G_i)|-1) \stackrel{(\ref{e:vertexsum})}{\leq} \Delta|V(G)|-1,\end{align*}
as desired.

    Thus we may assume that either there exists $e \in E(G)$ such that $\eta(e) \not \subseteq \nm{\mc{H}_i}$ for any $i \in [l]$. It follows that there exists $ j \in [l-1]$ such that $b_{H_j}(1)$ is an internal vertex of $\eta(e)$. Thus there exist vertex disjoint subgraphs $G',G''$ of $G$ such that $G \setminus e=G' \cup G''$, the restriction of $\eta$ to $G'$ is an embedding into $\mc{H}_{[j]}$, and  the restriction of $\eta$ to $G''$ is an embedding into $\mc{H}_{[j+1,l]}$. By the induction hypothesis we have $|E(G')| \leq \Delta|V(G')| -1$ and $|E(G'')| \leq \Delta|V(G'')| -1$, and so $|E(G)| \leq \Delta|V(G)|-1$, as desired.
\end{proof}	

The next lemma reduces the proof of \cref{t:topological}(ii) to finding a pair of graphs satisfying a few easily checkable properties. 

\begin{lem}\label{l:topdensity}
	If there exist graphs $H_{+}$ and $H_{-}$ such that $|V(H_{+})|,|V(H_{-})| \geq 2$, $\psi_{\Delta}(H_{+})~\geq~0,$  $\psi_{\Delta}(H_{-}) \leq 0$, and $\psi_{\Delta}(H_{+}) -  \psi_{\Delta}(H_{-}) \leq \Delta -1$, then
	$$d(\mc{F}_{\Delta})=\Delta.$$
\end{lem}

\begin{proof}
	We have $d(\mc{F}_{\Delta}) \leq \Delta$ by \cref{l:upper}. It remains to show that $d(\mc{F}_{\Delta}) \geq \Delta$.

	Let $\alpha = \psi_{\Delta}(H_{+})$ and $\beta =   \psi_{\Delta}(H_{-})$. Replacing $H_{+}$ and $H_{-}$ by their topological minors, if needed, we assume that $|E(H_{+})| = \Delta(|V(H_{+})|-1)+\alpha$, and $|E(H_{-})| = \Delta(|V(H_{-})|-1)+\beta$. 
	We convert the graphs $H_{+}$ and $H_{-}$ into $1$-patches by arbitrarily choosing distinct $a_{H_{+}}(1),b_{H_{+}}(1) \in V(H_{+})$ and $a_{H_{-}}(1),b_{H_{-}}(1) \in V(H_{-})$. 
	
	We now define an infinite sequence of $1$-patches $(H_i)_{i \in \bb{N}}$, such that $H_i \in \{H_{-},H_{+}\}$ for every $i \in \bb{N}$, as follows. Let $H_1=H_{-}$. Once $H_1,\ldots,H_k$ are defined, let $H_{k+1}=H_{-}$ if $\sum_{i \in [k]} \psi_{\Delta}(\nm{H_i}) \geq 0$, and let $H_{k+1}=H_{+}$, otherwise. 
	
	By induction on $k$ it immediately follows from the above definition that 
	\begin{equation}\label{e:totalexcess}
	 \beta \leq \sum_{i \in [k]} \psi_{\Delta}(\nm{H_i}) \leq \alpha,
	\end{equation}
for every $k \in \bb{N}$.
	Let $\mc{H}_l = (H_i)_{i \in [l]}.$ We claim that $\mc{H}_l$ is $\Delta$-controlled. Indeed, ($\Delta$2) holds by our definition of $1$-patches $H_{+}$ and $H_{-}$, and for  all $1 \leq j_1 \leq j_2 \leq l$  we have
	\begin{align*}
	\sum_{i=j_1}^{j_2}\psi_{\Delta}(\nm{H_i}) = \sum_{i \in [j_2]} \psi_{\Delta}(\nm{H_i}) - \sum_{i \in [j_1-1]} \psi_{\Delta}(\nm{H_i})\stackrel{\eqref{e:totalexcess}}{\leq} \alpha - \beta \leq \Delta -1, 
	\end{align*}
	and so ($\Delta$1)  also holds.
	Thus $\nm{\mc{H}_l} \in \mc{F}_{\Delta}$ for every $l$, and 
	\begin{align*}
	|E(\nm{\mc{H}_l})| &= \sum_{i \in [l]} |E(H_i)|  =  \sum_{i \in [l]}(\Delta(|V(H_i)|-1) + \psi_{\Delta}(\nm{H_i})) \\ &=
	\Delta(|V(\nm{\mc{H}_l})|-1) + \sum_{i \in [l]} \psi_{\Delta}(\nm{H_i})  \stackrel{\eqref{e:totalexcess}}{\geq} \Delta(|V(\nm{\mc{H}_l})|-1)  + \beta.
	\end{align*}
	Thus $\limsup_{l \rightarrow \infty} d(\nm{\mc{H}_l} )  \geq \Delta$, implying $d(\mc{F}_{\Delta}) \geq \Delta$.
\end{proof}

\begin{proof}[Proof of \cref{t:topological} (ii)] 
	It suffices for every $\Delta \geq 3/2$ to find the graphs $H_{+}$ and $H_{-}$ satisfying the conditions of \cref{l:topdensity}

Suppose first $\Delta \geq 2$.
	Let $t= \lceil 2\Delta \rceil$. Then $\psi_{\Delta}(K_t) \geq 0$, and $\psi_{\Delta}(G) < 0$ for every $G$ with $2 \leq |V(G)| < t$. Let $s = \lfloor \psi_{\Delta}(K_t) \rfloor$. Let $H_{+}$ and $H_{-}$ be obtained from $K_t$ by deleting $s$ and $s+1$ arbitrary edges, respectively. Then $\psi_{\Delta}(H_{+}) =  \psi_{\Delta}(K_t) - s \geq 0$, $0 > \psi_{\Delta}(H_{-}) \geq \psi_{\Delta}(H_{+}) -1$, and, consequently, $$\psi_{\Delta}(H_{+}) -  \psi_{\Delta}(H_{-}) \leq 1 \leq \Delta -1.$$ Thus $H_{+}$ and $H_{-}$ satisfy the desired conditions.
	
	It remains to consider the case $3/2  \leq \Delta < 2$. Let an integer $t \geq 3$ be chosen minimum such that $(2t-3) \geq \Delta(t-1)$.
	Let $H_{+}$ be the graph on $t$ vertices obtained from a path on $t-1$ vertices by adding a vertex adjacent to all the vertices of the path, and let $H_{-}$ be similarly obtained from a path on $t-2$ vertices. The graphs $H_{+}$ and $H_{-}$ are both outerplanar, and thus all topological minors of these graphs are outerplanar.
	It is well-known that $|E(H)| \leq 2|V(H)|-3$ for every outerplanar graph $H$ with $|V(H)| \geq 2$. It follows that
	$$\psi_{\Delta}(H_{+}) = \max_{2 \leq s \leq t}(2s-3 - \Delta(s-1))= (2t-3) - \Delta(t-1) \geq 0, $$ 	and, similarly, $\psi_{\Delta}(H_{-})= (2(t-1)-3)-\Delta(t-2) < 0$. Thus $$\psi_{\Delta}(H_{+}) - \psi_{\Delta}(H_{-}) = 2 - \Delta \leq \Delta - 1, $$
	as desired
\end{proof}

\section{Concluding remarks and open problems}\label{s:remarks}

\subsubsection*{Density vs. limiting density}
A natural, non-asymptotic notion of density of a graph class $\mc{F}$ is $$d^*(\mc{F})=\sup_{G \in \mc{F}, |V(G)|\geq 1}d(G).$$
If $\mc{F}$ is closed under taking disjoint unions then clearly $d(\mc{F})=d^*(\mc{F})$. In general, either  $d(\mc{F})=d^*(\mc{F})$, or there exists $G \in \mc{F}$ such that $d^*(\mc{F})=d(G)$. Thus \cref{t:rational} implies that
$d^*(\mc{F})$ is rational for every proper-minor closed graph class $\mc{F}$.

\subsubsection*{Extremal function vs. bounded pathwidth}
We don't know whether \cref{t:pathwidth} can be extended from densities to the extremal functions. 

\begin{qn}\label{que:maxnumberofedges}
	Let $\mc{F}$ be a proper minor-closed class of graphs. Does there necessarily exist a class $\mc{F}' \subseteq \mc{F}$ of bounded pathwidth such that $ex_{\mc{F}'}(n)=ex_{\mc{F}}(n)?$
\end{qn}

The first author~\cite[Theorem 6.5]{Kap18} proved that the answer to \cref{que:maxnumberofedges} is positive if $\mc{F}$ has bounded treewidth.

\subsubsection*{The complexity of $\mc{D}$}
Several questions from~\cite{Eppstein10} about the structure of $\mc{D}$ remain open, and it is possible that this structure rather involved. In particular, we do not know the answers to the following decidability questions.

\begin{qn}\label{que:densitydecidability}
	\begin{enumerate}
		\item Is the problem of the membership in $\mc{D}$ decidable? 
		\item Does there exist an algorithm which, given graphs $H_1,H_2,\ldots,H_k$ determines $d(\mc{F})$, where $\mc{F}$ is the class of all graphs containing none of $H_1,H_2,\ldots,H_k$ as a minor?   
	\end{enumerate}
\end{qn}

\subsubsection*{Topological minors}
It seems possible that the methods of this paper can be used to show that $\mc{D}_T \cap [0,3/2] \subseteq \bb{Q}$. We however did not pursue this, as, in contrast to $\mc{D}$,  complete understanding of the structure of $\mc{D}_T$ appears within reach.  Likely this can be accomplished via an extension of the methods used by Eppstein~\cite{Eppstein10} and by McDiarmid and Przykucki~\cite{McdPrz19} to explicitly describe $\mc{D} \cap [0,3/2]$ and $\mc{D} \cap [0,2]$, respectively.

\begin{problem}
	Give an explicit description of $\mc{D}_T \cap [0,3/2]$.
\end{problem}

\subsubsection*{Matroids}

Geelen, Gerards and Whittle~\cite{GeeGerWhi13}  conjectured generalizations of  Theorems~\ref{t:rational}--\ref{t:periodic} to linearly dense matroid classes. Many steps of the proof appear to  extend to matroids, but analogues of several other crucial ingredients, in particular of \cref{t:RSwqo} and \cref{t:flat} are  missing.

\subsubsection*{Acknowledgements} 
We thank Jim Geelen, Paul Seymour and Robin Thomas for stimulating discussions on the subject of this paper.
 
\bibliographystyle{plain}
\bibliography{../snorin}
\end{document}